\documentclass[11pt,leqno]{amsart}
\usepackage{verbatim,amssymb,amsfonts,latexsym,amsmath,amsthm,tikz,bbm,nicefrac,xspace,graphicx,mathtools,mathdots,enumerate}
\usetikzlibrary{arrows}
\newtheorem{theorem}{Theorem}[section]
\newtheorem{lemma}[theorem]{Lemma}
\newtheorem{corollary}[theorem]{Corollary}

\newtheorem{question}[theorem]{Question}
\newtheorem*{maintheorem1}{Main Theorem}
\newtheorem*{maintheorem2}{Main Theorem (dual form)}

\newtheorem*{theorem34}{Theorem~\ref{thm:main}}
\newtheorem*{theorem*}{Theorem}
\newtheorem*{theoremdual}{Dual form of Theorem~\ref{thm:main1}}
\newtheorem*{corollary*}{Corollary}
\newtheorem*{observation1}{Observation 1}
\newtheorem*{observation2}{Observation 2}
\newtheorem*{observation3}{Observation 3}
\newtheorem*{claim}{Claim}
\newtheorem*{sub-claim}{sub-claim}
\theoremstyle{definition}
\newtheorem{example}[theorem]{Example}
\newtheorem{definition}[theorem]{Definition}

\theoremstyle{remark}

\newcommand{\R}{\mathbb{R}}
\newcommand{\N}{\mathbb{N}}
\newcommand{\Z}{\mathbb{Z}}
\newcommand{\U}{\mathcal{U}}

\renewcommand{\O}{\mathcal{O}}
\newcommand{\explicitSet}[1]{\left\lbrace #1 \right\rbrace}
\newcommand{\brackets}[1]{\left\langle #1 \right\rangle}
\newcommand{\set}[2]{\explicitSet{#1 \colon #2}}
\newcommand{\seq}[2]{\brackets{#1 \colon #2}}
\newcommand{\<}{\langle}
\renewcommand{\>}{\rangle}
\renewcommand{\a}{\alpha}
\renewcommand{\b}{\beta}
\newcommand{\g}{\gamma}
\newcommand{\dlt}{\delta}
\newcommand{\e}{\varepsilon}
\newcommand{\z}{\zeta}
\renewcommand{\k}{\kappa}
\newcommand{\s}{\sigma}
\renewcommand{\t}{\tau}
\newcommand{\w}{\omega}
\newcommand{\0}{\emptyset}
\newcommand{\sub}{\subseteq}
\newcommand{\rest}{\!\restriction\!}

\newcommand{\iso}{\cong}
\newcommand{\closure}[1]{\overline{#1}}

\newcommand{\ulim}{\U\mbox{-}\!\lim_{n \in D}}
\newcommand{\phiulim}{\varphi^*_g(\U)\mbox{-}\!\lim_{n \in D}}

\newcommand{\gulim}{g^*(\U)\mbox{-}\!\lim_{n \in D}}

\newcommand{\ulimmm}{\U\mbox{-}\!\lim_{n \in \Z}}
\newcommand{\invlim}{\textstyle \varprojlim}
\newcommand{\quotient}{\twoheadrightarrow}
\newcommand{\haleph}{H_{\aleph_2}}
\newcommand{\pwmf}{\mathcal{P}(\w)/\mathrm{fin}}

\newcommand{\up}{^\uparrow}
\renewcommand{\AA}{\mathbb{A}}
\newcommand{\BB}{\mathbb{B}}

\newcommand{\continuum}{\mathfrak{c}}

\newcommand{\ch}{\ensuremath{\mathsf{CH}}\xspace}
\newcommand{\zfc}{\ensuremath{\mathsf{ZFC}}\xspace}
\newcommand{\pfa}{\ensuremath{\mathsf{PFA}}\xspace}
\newcommand{\ocama}{\ensuremath{\mathsf{OCA+MA}}\xspace}
\newcommand{\ma}{\ensuremath{\mathsf{MA}}\xspace}

\newcommand{\oca}{\ensuremath{\mathsf{OCA}}\xspace}
\newcommand{\cbul}{\,^{\!\,_{\!\,_{\!\,_{\bullet}}}}}
\newcommand{\ccirc}{\,^{\!\,_{\!\,_{\!\,_{\circ}}}}}

\newcommand{\mapstoof}{\xmapsto{\ \!_{\,\!_{f}}\,}_{\!_{\O}}}
\newcommand{\mapstoo}{\xmapsto{\ \!_{\,\!_{\psi_p}}\,}_{\!_{\O}}}
\newcommand{\mapstooof}{\xmapsto{\ \!_{\,\!_{\psi_p}}\,}_{\!_{\O^m(E)}}}
\newcommand{\mapstoooof}{\xmapsto{\ \!_{\,\!_{\psi_{p_j}}}\,}_{\!_{\O^m(E^m)}}}
\newcommand{\mapstoooaf}{\xmapsto{\ \!_{\,\!_{\psi_{p_j}}}\,}_{\!_{\O^m(F)}}}
\newcommand{\mapstooobf}{\xmapsto{\ \!_{\,\!_{\psi_{p_j}}}\,}_{\!_{\O^m(E)}}}
\newcommand{\mapstooocf}{\xmapsto{\ \!_{\,\!_{\psi_{p_j}}}\,}_{\!_{\O^m(F^\uparrow_m)}}}
\newcommand{\mapstooodf}{\xmapsto{\ \!_{\,\!_{\psi_{p_j}}}\,}_{\!_{\O^m}}}
\newcommand{\mapstoooef}{\xmapsto{\ \!_{\,\!_{\psi_{p_j}}}\,}_{\!_{\O^{m'}}}}%%

\newcommand{\embeds}{\hookrightarrow}
\renewcommand{\quotient}{\twoheadrightarrow}

\begin{document}

%%%%%%%%%%%%
\title{Universal flows and automorphisms of $\pwmf$}
%%%%%%%%%%%%
\author{Will Brian}
\address {
W. R. Brian\\
Department of Mathematics and Statistics\\
University of North Carolina at Charlotte\\
9201 University City Blvd.\\
Charlotte, NC 28223-0001}
\email{wbrian.math@gmail.com}
\urladdr{wrbrian.wordpress.com}
%%%%%%%%%%%%
\subjclass[2010]{Primary: 54H20, 03E35, Secondary: 03C98, 37B05}
\keywords{Stone-\v{C}ech compactification, flows and dynamical systems, universal properties, Continuum Hypothesis, elementary submodels}
%\thanks{}
%%%%%%%%%%%%

%%%%%%%%%%%%
\begin{abstract}
We prove that for every countable discrete group $G$, there is a $G$-flow on $\w^*$ that has every $G$-flow of weight $\leq\! \aleph_1$ as a quotient. It follows that, under the Continuum Hypothesis, there is a universal $G$-flow of weight $\leq\!\continuum$.

Applying Stone duality, we deduce that, under \ch, there is a trivial automorphism $\t$ of $\pwmf$ with every other automorphism embedded in it, which means that every other automorphism of $\pwmf$ can be written as the restriction of $\t$ to a suitably chosen subalgebra.
\end{abstract}
%%%%%%%%%%%%

\maketitle

%%%%%%%%%%%%

\section{Introduction}

A $G$\emph{-flow} is a continuous action of a group $G$ on a compact Hausdorff space. Given two $G$-flows $\cbul\!: G \times X \to X$ and $\ccirc\!: G \times Y \to Y$, a \emph{quotient map} from $\cbul$ to $\ccirc$ is a continuous surjection $\pi: X \to Y$ that preserves the action of $G$, in the sense that $\pi(g \cbul x) = g \ccirc \pi(x)$ for every $g$ and $x$.

%A $G$-flow is \emph{universal} with respect to some property $P$ if it has property $P$, and if every other $G$-flow with property $P$ is a quotient of it.

\begin{maintheorem1}
Let $G$ be a countable discrete group, and consider the trivial flow on the space $(G \times \w)^*$ induced by the natural action of $G$ on $G \times \omega$, namely $g \cbul (h,n) = (gh,n)$.
\begin{enumerate}
\item Every $G$-flow of weight $\leq\! \aleph_1$ is a quotient of this flow.
\item Consequently, the Continuum Hypothesis implies this flow is universal for $G$-flows of weight $\leq\! \continuum$.
\end{enumerate}
\end{maintheorem1}

\noindent Here $(G \times \w)^* = \beta(G \times \w) \setminus (G \times \w)$ is the Stone-\v{C}ech remainder of the discrete space $G \times \w$. The \emph{weight} of a flow means the weight of its underlying topological space $X$, that is, the smallest cardinality of a basis for $X$.

What is called the ``Main Theorem'' above is in fact just a corollary to a stronger result, which we call the ``Main Lemma,'' but its statement is more technical and will be postponed until later. Roughly, the main lemma characterizes the weight-$\aleph_1$ quotients of any trivial $S$-flow on $\w^*$ for any countable discrete semigroup $S$. This improves on some former work of the author in \cite{Brian}, which gives a topological characterization of the weight-$\aleph_1$ quotients of a particular $(\N,+)$-flow on $\w^*$ (the one generated by the shift map).

The assumption of the Continuum Hypothesis (henceforth abbreviated \ch) in part (2) of this theorem is necessary, in the sense that the conclusion does not follow from \zfc alone. Even more is true: \zfc does not prove the existence of universal weight-$\continuum$ flows for any group, as the following observation shows.

\begin{theorem}\label{thm:noflows}
It is consistent that no (semi)group $G$ admits a universal $G$-flow of weight $\leq\! \continuum$.
\end{theorem}
\begin{proof}
Observe that for every (semi)group $G$ and every compact Hausdorff space $X$, there is at least one $G$-flow on $X$, namely the trivial flow $(g,x) \mapsto x$.
It follows from this observation that if $\cbul\!: G \times Y \to Y$ is a universal $G$-flow of weight $\leq\! \continuum$, then every compact Hausdorff space $X$ of weight $\leq\! \continuum$ is a continuous image of $Y$.
In Section 6 of \cite{Dow&Hart}, Dow and Hart show that it is consistent that no such space $Y$ exists.
\end{proof}

\vspace{2mm}

If $G = (\Z,+)$, then each $G$-flow on $\w^*$ is generated by an autohomeomorphism of $\w^*$ and, conversely, every autohomeomorphism of $\w^*$ generates a $G$-flow. Thus, via Stone duality, results concerning $(\Z,+)$-flows on $\w^*$ can be translated into results concerning automorphisms of the Boolean algebra $\pwmf$.

Given two automorphisms $\varphi,\psi: \pwmf \to \pwmf$, we say that $\varphi$ \emph{embeds} in $\psi$ if there is a self-embedding $e: \pwmf \to \pwmf$ such that $e \circ \varphi = \psi \circ e$. Equivalently, $\varphi$ embeds in $\psi$ if there is a subalgebra $\AA$ of $\pwmf$ such that $(\pwmf,\varphi)$ is isomorphic to $(\AA,\psi \rest \AA)$. 

\begin{center}
\begin{tikzpicture}

\node at (5,0) {$\pwmf$};
\node at (8.5,0) {$\pwmf$};
\node at (8.5,2) {$\pwmf$};
\node at (5,2) {$\pwmf$};
\draw[->] (5.95,2) -- (7.55,2); \node at (6.75,2.25) {$\psi$};
\draw[->] (5.95,0) -- (7.55,0); \node at (6.75,-.25) {$\varphi$};
\draw[right hook->] (8.5,.35) -- (8.5,1.68); \node at (8.7,1) {\small $e$};
\draw[right hook->] (5,.35) -- (5,1.68); \node at (4.8,1) {\small $e$};
\node at (11.5,1.25) {\Large $\varphi \embeds \psi$};

\end{tikzpicture}
\end{center}

\begin{maintheorem2}
Assuming \ch, there is a trivial automorphism $\t$ of the Boolean algebra $\pwmf$ such that every other automorphism embeds in $\t$. 
\end{maintheorem2}

By applying Stone duality, this result follows directly from a special case of the Main Theorem, namely the case $G = (\Z,+)$. The automorphism $\t$ mentioned in the theorem is the one generated by the map $(z,n) \mapsto (z+1,n)$ on $\Z \times \w$.

In fact, we will go a bit further and investigate precisely which automorphisms of $\pwmf$ are universal in the sense of this dualized Main Theorem. The investigation is carried out in \zfc+\ch, because those are the axioms needed for applyling the Main Lemma. We will classify precisely which trivial automorphisms are universal under \ch, and show that there are many nontrivial universal automorphisms as well.

\vspace{2mm}

The paper is organized as follows.

Section~\ref{sec:theorem} begins by establishing the terminology necessary for stating the Main Lemma, and ends by stating it and deriving the Main Theorem from it. We will include a brief outline of the proof, and some hints as to where the difficulties lie, but the proof of this result is put off until the end of the paper.

Section~\ref{sec:dynamics} looks at the automorphisms of $\pwmf$ in light of the Main Lemma, and initiates an investigation of the ``embeds in'' relation on the set of automorphisms of $\pwmf$, focusing on universal automorphisms. Most of the theorems in Section~\ref{sec:dynamics} have relatively simple proofs that use the Main Lemma, Stone duality, and basic ideas and techniques from topological dynamics. Most of the results are established in \zfc+\ch so that the Main Lemma can be applied; but some observations using \ocama are made in order to establish the independence of some of the results proved using \ch. Section~\ref{sec:dynamics} includes several open questions.

%Though we have stated the result for Boolean algebras, automorphisms, and embeddings in the introduction, the proof takes place in the dual category of compact Hausdorff spaces, autohomeomorphisms, and quotient mappings.

At last, Section~\ref{sec:proof} gives a proof of the Main Lemma. A significant portion of this section is devoted to introducing and developing the topological tools needed to prove the result. These tools are used in the proof of this result and nowhere else in the paper, and this is part of the reason for postponing the proof until the end of the paper. The other part is the author's opinion that the material in Sections \ref{sec:dynamics} is simply more exciting.

\section{The main lemma}\label{sec:theorem}

\subsection{Semigroup actions and Stone-\v{C}ech remainders}

%A semigroup $S$ is \emph{cancellative} if for all $p \in S$, the functions $x \mapsto px$ and $x \mapsto xp$ are injective. Main Theorem 1 was stated for groups in the introduction, but will be proved for the broader class of countable, discrete, cancellative semigroups.

An \emph{action} of a semigroup $S$ on a set $X$ is a function $\varphi: S \times X \to X$ with the property that $\varphi(q,\varphi(p,x)) = \varphi(pq,x)$ for all $p,q \in S$. 
We will almost always write $\varphi_p(x)$ for $\varphi(p,x)$. Using this notation, an action of $S$ on $X$ is a function $\varphi: S \times X \to X$ such that $\varphi_p \circ \varphi_q = \varphi_{pq}$ for all $p,q \in S$. An action is called \emph{separately finite-to-one} if each of the functions $\varphi_p$ is finite-to-one.

All the semigroups we consider in this paper are discrete. Thus an action $\varphi$ of a semigroup $S$ on a topological space $X$ is called \emph{continuous} if each of the functions $\varphi_p$ is continuous. (Because $S$ is discrete, this is equivalent to the requirement that $\varphi$ be continuous on $S \times X$.)
Thus an $S$-flow is simply a collection $\set{\varphi_p}{p \in S}$ of continuous functions on a compact Hausdorff space $X$, with the property that $\varphi_p \circ \varphi_q = \varphi_{pq}$ for all $p,q \in S$. In what follows, we will often consider an $S$-flow to be a $S$-indexed collection of functions on $X$, rather than a single function on $S \times X$.

If $D$ is a countable discrete space, then $\b D$ and $D^*$ are, respectively, the Stone-\v{C}ech compactification and the Stone-\v{C}ech remainder of $D$. In what follows we will usually have $D = \w$, $D = S$, or $D = S \times \w$ for some countable discrete semigroup $S$.

Every function $f: D \to D$ extends (uniquely) to a continuous function $\b f: \b D \to \b D$, called the \emph{Stone extension} of $f$, defined by
$$\b f(\U) = \set{A \sub D}{f^{-1}(A) \in \U}$$
for all $\U \in \b D$..
If $f: D \to D$ is finite-to-one, then $\b f$ maps $D^*$ to $D^*$ and we denote $\b f \rest D^*$ by $f^*$.
We say that $f^*$ is \emph{induced} by $f$, and any function on $D^*$ that is obtained in this way is called \emph{trivial}.

Suppose $\varphi: S \times D \to D$ is a separately finite-to-one action of a semigroup $S$ on a countable discrete set $D$. For each $p \in S$ the function $\varphi_p: D \to D$ induces a trivial map $(\varphi_p)^*: D^* \to D^*$.
Taken together, these functions define an action of $S$ on $D^*$.
Formally, we define a function $\varphi^*: S \times D^* \to D^*$ by setting
$$(\varphi^*)_p = (\varphi_p)^*$$
for all $p \in S$ (and naturally, we write $\varphi^*_p$ instead of $(\varphi_p)^*$ or $(\varphi^*)_p$). One may easily check that $\varphi^*$ is an action of $S$ on $D^*$, and, because each of the functions $\varphi_p^*$ is continuous, this makes $\varphi^*$ an $S$-flow. Any $S$-flow on $D^*$ arising in this way is called \emph{trivial}, and we say that the action $\varphi^*$ on $D^*$ is \emph{induced} by the action $\varphi$ on $D$.

\subsection{Metrizable reflections.}

Recall that $H_\k$ denotes the set of all sets hereditarily smaller than $\k$. In what follows, the letter $H$ will always denote a set of the form $H_\k$ for some sufficiently large value of $\k$. The precise value of $\k$ does not matter very much for our purposes, but the reader who wishes to be economical is welcome to verify that $\k = \aleph_2$ is sufficient, because every structure we consider can be coded inside $\haleph$.

To be more precise, let us described a way of coding a weight-$\aleph_1$ flow in $\haleph$. Every topological space of weight $\aleph_1$ is (up to homeomorphism) a closed subspace of $[0,1]^{\aleph_1}$. A ``code'' for a closed subspace of $[0,1]^{\aleph_1}$ could be an $\w_1$-length list of all the basic open neighborhoods in $[0,1]^{\aleph_1}$ that intersect $X$. A continuous function on $X$ could be coded as a directed graph, with these basic open neighborhoods as its vertices, the arrows in the graph indicating the relation $f(U \cap X) \cap V \neq \0$.

We will look at countable elementary submodels of the structure $(H,\in)$. Roughly, the idea is that if $(M,\in) \preceq (H,\in)$ then only a countable portion of the ``code'' for some weight-$\aleph_1$ flow is captured by $M$. Decoding this flattened fragment of the code leads to a metrizable flow, which we call a metrizable reflection of the original. The process is not unlike file compression: an elementary submodel acts as a data compressor on our topological flows, turning a file with $\aleph_1$ bits of information into a zip file with $\aleph_0$ bits that still carries essentially the same message.

In what follows, we will work at a topological level and avoid the tedious business of encoding and decoding spaces. On the one hand, this ``pointed'' approach forces us to work with larger objects than necessary, which obscures the fact that our theorems are naturally set in $\haleph$; on the other hand, we will see that points and sequences of points feature prominently in our proofs, and eliminating them would obscure the ideas behind the proofs.

If $(M,\in)$ is a countable elementary submodel of $(H,\in)$, then $M \cap \w_1$ is a countable ordinal number. We denote this ordinal by $\dlt^M$.

The points of $[0,1]^{\aleph_1}$ are functions $x: \w_1 \to [0,1]$, and we define the projection maps
$$\pi_\a(x) = x(\a) \qquad \text{ and } \qquad \Pi_\dlt(x) = x \rest \dlt$$
for $\a,\dlt \in \w_1$. 
If for each $\a < \dlt$ we have a function $x_\a: Z \to [0,1]$, then $\Delta_{\a < \dlt}x_\a$ denotes the unique function $Z \to [0,1]^\dlt$ with $\pi_\a \circ \Delta_{\a < \dlt}x_\a = x_\a$ for all $\a < \w_1$.

\begin{definition}$\ $
\begin{itemize}
\item Let $X$ be a closed subspace of $[0,1]^{\w_1}$, and let $M$ be a countable elementary submodel of $H$. The \emph{reflection of $X$ in $M$} is the space
$$X^M = \Pi_{\dlt^M}[X].$$
Any space of this form is called a \emph{metrizable reflection} of $X$.

\item If $f: X \to X$ is continuous, then the reflection of $f$ in $M$, denoted $f^M$, is the continuous function $f^M: X^M \to X^M$ defined by the equation
$$\Pi_{\dlt^M} \circ f = f^M \circ \Pi_{\dlt^M},$$
In other words, $f^M$ is the unique continuous function making the following diagram commute:
\begin{center}
\begin{tikzpicture}
\draw [fill=white,white] (-1,0) circle (5pt);  \node at (-1,0) {$X^M$};
\draw [fill=white,white] (2,0) circle (5pt);  \node at (2,0) {$X^M$};
\draw [fill=white,white] (2,2) circle (5pt);  \node at (2,2) {$X$};
\draw [fill=white,white] (-1,2) circle (5pt);  \node at (-1,2) {$X$};
\draw[->] (-.6,2) -- (1.6,2); \node at (.5,2.3) {$f$};
\draw[->,dashed] (-.6,0) -- (1.6,0); \node at (.5,-.3) {$f^M$};
\draw[->] (2,1.68) -- (2,.3); \node at (2.45,1) {\small $\Pi_{\dlt^M}$};
\draw[->] (-1,1.68) -- (-1,.3); \node at (-1.4,1) {\small $\Pi_{\dlt^M}$};
\end{tikzpicture}
\end{center}
Any function of this form is called a \emph{metrizable reflection} of $f$.

\item Similarly, if $\psi$ is an $S$-flow on $X$, then the \emph{reflection of $\psi$ in $M$} is denoted $\psi^M$ and is defined by putting $(\psi^M)_p = (\psi_p)^M$ for all $p \in S$ (and, of course, this map will be denoted simply by $\psi^M_p$). Any flow of this form is called a \emph{metrizable reflection} of $\psi$.
\end{itemize}
\end{definition}

This definition is essentially due to Bandlow; see \cite{Bandlow}. It is not immediately clear that the function $f^M$ is well-defined, but this turns out to be a consequence of elementarity, articulated in the following lemma.

\begin{lemma}
Let $X$ be a subspace of $[0,1]^{\w_1}$ and let $f: X \to X$ be continuous. If $M$ is an elementary submodel of $H$, then
$$\Pi_{\dlt^M}(x) = \Pi_{\dlt^M}(y) \qquad \text{implies} \qquad \Pi_{\dlt^M} \circ f(x) = \Pi_{\dlt^M} \circ f(y)$$
for all $x,y \in X$. In other words, the function $f^M: X^M \to X^M$ is well-defined.
\end{lemma}

This lemma was first proved by Noble and Ulmer in \cite{N&U} and later rediscovered by Shchepin in \cite{Shchepin} (neither source phrases it quite this way, but their proofs show this nonetheless).

Notice that if $\psi$ is a flow, then any of its metrizable reflections $\psi^M$ is a quotient of $\psi$, with $\Pi_{\dlt^M}$ being the natural quotient mapping.

\subsection{Sequences that almost match a flow.}

A sequence $\seq{x_n}{n \in D}$ of points in a space $X$ is called \emph{tail-dense} if for any finite $F \sub D$, the set $\set{x_n}{n \in D \setminus F}$ is dense in $X$.
If $X$ is a metric space and $x,y \in X$, we write $x \approx_\e y$ to mean $\mathrm{dist}(x,y) < \e$.

\begin{definition}\label{def:philike}
Let $X$ be a compact metric space and $D$ a countable set.
\begin{itemize}
\item Let $f: X \to X$ be continuous, and let $p: D \to D$ be any function.
\begin{itemize}
\item[$\circ$] A $D$-indexed sequence $\seq{x_n}{n \in D}$ of points in $X$ is said to be \emph{$p$-like up to $\e$ with respect to $f$}
if
$$f(x_n) \approx_\e x_{p(n)} \qquad \text{for all but finitely many } n \in D.$$
When $f$ is clear from context, we say simply that the sequence is $p$-like up to $\e$.
\item[$\circ$] If $\seq{x_n}{n \in D}$ is $p$-like up to $\e$ for every $\e > 0$, then we say that $\seq{x_n}{n \in D}$ is \emph{$p$-like (with respect to $f$)}.
\end{itemize}
\item Let $\psi: S \times X \to X$ be an $S$-flow, and let $\varphi: S \times D \to D$ be an action of $S$ on $D$.
\begin{itemize}
\item[$\circ$] A $D$-indexed sequence $\seq{x_n}{n \in D}$ of points in $X$ is said to be \emph{$\varphi$-like up to $\e$ with respect to $\psi$} if  for every $p \in S$, it is $\varphi_p$-like up to $\e$ with respect to $\psi_p$. When $\psi$ is clear from context, we say simply that the sequence is $p$-like up to $\e$.
\item[$\circ$] If $\seq{x_n}{n \in D}$ is $\varphi$-like up to $\e$ for every $\e > 0$, then we say that $\seq{x_n}{n \in D}$ is $\varphi$-like (with respect to $f$).
\end{itemize}
\end{itemize}
\end{definition}

In other words, $\seq{x_n}{n \in D}$ is $\varphi$-like up to $\e$ if, on a tail of the sequence, the action of $\psi_p$ on the $x_n$ ``approximately'' (up to an error of $\e$) matches the action of $\varphi_p$ on their indices. The first part of the definition is illustrated below.

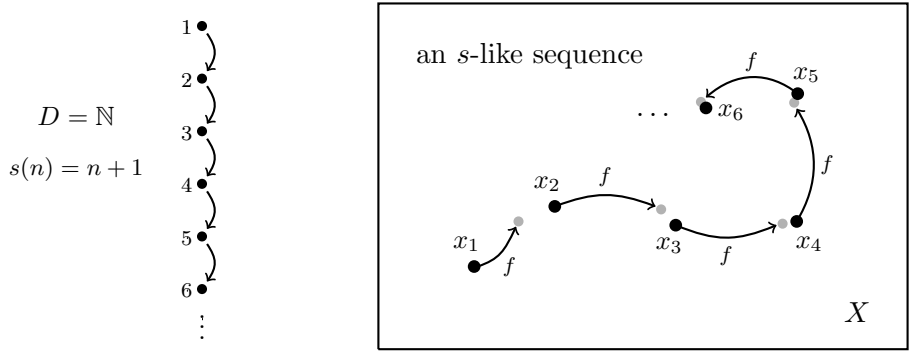
\vspace{2mm}\begin{center}
\begin{tikzpicture}[style=thick, xscale=.67,yscale=1]

\draw (-5,-.6) node [circle, draw, fill=black, inner sep=0pt, minimum width=3pt] {} -- (-5,-.6);
\draw (-5,.1) node [circle, draw, fill=black, inner sep=0pt, minimum width=3pt] {} -- (-5,.1);
\draw (-5,.8) node [circle, draw, fill=black, inner sep=0pt, minimum width=3pt] {} -- (-5,.8);
\draw (-5,1.5) node [circle, draw, fill=black, inner sep=0pt, minimum width=3pt] {} -- (-5,1.5);
\draw (-5,2.2) node [circle, draw, fill=black, inner sep=0pt, minimum width=3pt] {} -- (-5,2.2);
\draw (-5,-1.3) node [circle, draw, fill=black, inner sep=0pt, minimum width=3pt] {} -- (-5,-1.3);
\draw[->] (-4.9,2.1) .. controls (-4.7,1.9) and (-4.7,1.8) .. (-4.9,1.6);
\draw[->] (-4.9,1.4) .. controls (-4.7,1.2) and (-4.7,1.1) .. (-4.9,.9);
\draw[->] (-4.9,.7) .. controls (-4.7,.5) and (-4.7,.4) .. (-4.9,.2);
\draw[->] (-4.9,0) .. controls (-4.7,-.2) and (-4.7,-.3) .. (-4.9,-.5);
\draw[->] (-4.9,-.7) .. controls (-4.7,-.9) and (-4.7,-1) .. (-4.9,-1.2);
\draw (-5,-1.7) node {\small $\vdots$} -- (-5,-1.7);
\draw (-5.3,2.18) node {\scriptsize $1$} -- (-5.3,2.18);
\draw (-5.3,1.48) node {\scriptsize $2$} -- (-5.3,1.48);
\draw (-5.3,.78) node {\scriptsize $3$} -- (-5.3,.78);
\draw (-5.3,.08) node {\scriptsize $4$} -- (-5.3,.08);
\draw (-5.3,-.62) node {\scriptsize $5$} -- (-5.3,-.62);
\draw (-5.3,-1.32) node {\scriptsize $6$} -- (-5.3,-1.32);
\node at (-7.5,1) {\small $D = \N$};
\node at (-7.5,.3) {\footnotesize $s(n) = n+1$};

\draw (-1.5,-2.1) -- (-1.5,2.5) -- (9,2.5) -- (9,-2.1) -- (-1.5,-2.1);
\node at (8,-1.6) {$X$};
\node at (1.5,1.8) {an $s$-like sequence};

\node[circle, draw, fill=black, inner sep=0pt, minimum width=4pt] at (.4,-1) {};
\node at (.25,-.7) {\small $x_1$};
\draw[->] (.47,-1) .. controls (.9,-.9) and (1,-.8) .. (1.23,-.47);
\node at (1.1,-1) {\scriptsize $f$};
\node[circle, black!30, draw, fill=black!30, inner sep=0pt, minimum width=3pt] at (1.28,-.4) {};%%%%%
\node[circle, draw, fill=black, inner sep=0pt, minimum width=4pt] at (2,-.2) {};
\node at (1.85,.1) {\small $x_2$};
\draw[->] (2,-.2) .. controls (2.8,0) and (3.2,0) .. (4,-.2);
\node at (3,.2) {\scriptsize $f$};
\node[circle, black!30, draw, fill=black!30, inner sep=0pt, minimum width=3pt] at (4.11,-.24) {};%%%%%
\node[circle, draw, fill=black, inner sep=0pt, minimum width=4pt] at (4.4,-.45) {};
\node at (4.27,-.72) {\small $x_3$};
\draw[->] (4.4,-.45) .. controls (5.2,-.67) and (5.6,-.67) .. (6.4,-.45);
\node at (5.4,-.85) {\scriptsize $f$};
\node[circle, black!30, draw, fill=black!30, inner sep=0pt, minimum width=3pt] at (6.52,-.43) {};%%%%%
\node[circle, draw, fill=black, inner sep=0pt, minimum width=4pt] at (6.8,-.4) {};
\node at (7.05,-.7) {\small $x_4$};
\draw[->] (6.8,-.4) .. controls (7.25,.1) and (7.25,.6) .. (6.8,1.1);
\node at (7.4,.35) {\scriptsize $f$};
\node[circle, black!30, draw, fill=black!30, inner sep=0pt, minimum width=3pt] at (6.75,1.18) {};%%%%%
\node[circle, draw, fill=black, inner sep=0pt, minimum width=4pt] at (6.82,1.3) {};
\node at (7,1.55) {\small $x_5$};
\draw[->] (6.82,1.3) .. controls (6.25,1.6) and (5.6,1.6) .. (5,1.25);
\node at (5.9,1.73) {\scriptsize $f$};
\node[circle, black!30, draw, fill=black!30, inner sep=0pt, minimum width=3pt] at (4.9,1.19) {};%%%%%
\node[circle, draw, fill=black, inner sep=0pt, minimum width=4pt] at (5,1.11) {};
\node at (5.5,1.05) {\small $x_6$};
\node at (4,1.02) {$\dots$};

\end{tikzpicture}\end{center}\vspace{2mm}

As the picture indicates, for the successor map $s(n) = n+1$, an $s$-like sequence is simply a sequence $\seq{x_n}{n \in \N}$ of points in $X$ such that $\lim_{n \to \infty} \mathrm{dist}(f(x_n),x_{n+1}) = 0$. Such sequences are sometimes called ``asymptotic pseudo-orbits'' in the topological dynamics literature.

\subsection{The main lemma and the main theorem.}

We are finally in a position to state the full version of the main result of this paper:

\begin{theorem}[the Main Lemma]\label{thm:main}
Let $S$ be a countable discrete semigroup, let $X$ be a compact Hausdorff space of weight $\leq\!\aleph_1$, and let $D$ be a countable set. Let $\varphi$ be a separately finite-to-one action of $S$ on $D$, and let $\psi: S \times X \to X$ be an $S$-flow. The following are equivalent:
\begin{enumerate}
\item $\psi$ is a quotient of $\varphi^*$.
\item Every metrizable quotient of $\psi$ is a quotient of $\varphi^*$.
\item Some metrizable reflection of $\psi$ is a quotient of $\varphi^*$.
\item Every metrizable quotient of $\psi$ contains a tail-dense $\varphi$-like sequence.
\item Some metrizable reflection of $\psi$ contains a tail-dense $\varphi$-like sequence.
\end{enumerate}
\end{theorem}

As we will see in the following two sections, if we are given a flow $\psi$ and an action $\varphi: D \to D$, it is often very easy to decide when $(4)$ holds. Thus this theorem enables us to decide, with relative ease, when a weight-$\aleph_1$ flow is the quotient of a trivial flow on $\w^*$.

\begin{theorem}[the Main Theorem]\label{thm:main1}
Let $G$ be a countable discrete group, and consider the trivial flow $\varphi^*$ on the space $(G \times \w)^*$ induced by the natural action $\varphi$ of $G$ on $G \times \omega$, namely $\varphi_g(h,n) = (gh,n)$.
\begin{enumerate}
\item Every $G$-flow of weight $\leq\! \aleph_1$ is a quotient of $\varphi^*$.
\item Assuming \ch, $\varphi^*$ is a universal $G$-flow of weight $\leq\! \continuum$.
\end{enumerate}
\end{theorem}
\begin{proof}
By Theorem~\ref{thm:main}, it suffices to show that every metrizable $G$-flow contains a tail-dense $\varphi$-like sequence of points.

Suppose $\psi: G \times X \to X$ is a metrizable $G$-flow. Let $\seq{d_n}{n \in \w}$ be a tail-dense sequence of points in $X$. Define a sequence $\seq{x_{(g,n)}}{(g,n) \in G \times \w}$ of points in $X$ by putting
$$x_{(g,n)} = \psi_g(d_n)$$
for all $g \in G$ and $n \in \w$. This sequence is obviously $\varphi$-like. Note that, if $e$ is the identity element of $G$, $x_{(e,n)} = d_n$ for all $n$; thus, by our choice of the $d_n$, this sequence is tail-dense.
\end{proof}

We will see in the next section that the proof of this result extends to the semigroup $(\N,+)$ as well. For the time being, non-discrete groups and uncountable groups both seem to be out of reach.

\begin{question}
Is it consistent to have a universal $(\R,+)$-flow of weight $\continuum$? Is the natural flow on $(\R \times \w)^*$ an example under \ch?
\end{question}

\subsection{A few comments on the proof}

By forgetting about groups and actions (i.e., applying the ``forgetful functor'' mapping the category of $G$-flows to the category of compact Hausdorff spaces, as in the proof of Theorem~\ref{thm:noflows}), the Main Theorem reduces to Parovi\v{c}enko's Theorem \cite{Parovicenko}, a classic result of set-theoretic topology:
\begin{itemize}
\item Every compact Hausdorff space of weight $\leq\! \aleph_1$ is a continuous image of $\w^* = \b\w \setminus \w$.
\end{itemize}

One may view Theorem~\ref{thm:main} as the ``dynamic version'' of Parovi\v{c}enko's theorem, its natural analogue in the category of $G$-flows, which identifies (under \ch) a natural universal object for the category. It is natural, then, that the proof of our main theorem shares some features with a proof of Parovi\v{c}enko's theorem.

Of the various proofs of Parovi\v{c}enko's theorem, ours is closest in spirit to that of B{\l}aszczyk and Szyma\'nski in \cite{B&S}. Their proof begins by writing a given compact Hausdorff space $X$ as a length-$\omega_1$ inverse limit of compact metrizable spaces: $X = \varprojlim \seq{X_\a}{\a < \w_1}$. They then construct a coherent transfinite sequence of continuous surjections $\pi_\a: \w^* \to X_\a$, and define $\pi: \w^* \to X$ to be the inverse limit of this sequence. The $\pi_\a$ are constructed recursively, using a variant of the following lifting lemma at successor stages:

\begin{lemma}\label{lem:easylift}
Let $Z$ and $Y$ be compact metric spaces, and let $f: Z \to Y$ be a continuous surjection. If $\pi_Y: \w^* \to Y$ is a continuous surjection, then it can be lifted to a continuous surjection $\pi_Z: \w^* \to Z$ such that $\pi_Y = f \circ \pi_Z$.
\end{lemma}
\vspace{-1mm}
\begin{center}
\begin{tikzpicture}
\draw [fill=white,white] (-1,0) circle (5pt);  \node at (-1,0) {$Y$};
\draw [fill=white,white] (2,0) circle (5pt);  \node at (2,0) {$Z$};
\draw [fill=white,white] (2,2) circle (5pt);  \node at (2,2) {$\ \w^*$};
\draw[->] (1.65,0) -- (-.65,0); \node at (.5,-.25) {$f$};
\draw[->] (1.75,1.8) -- (-.7,.18); \node at (.5,1.3) {$\pi_Y$};
\draw[->,dashed] (2,1.68) -- (2,.3); \node at (2.3,1) {$\pi_Z$};
\end{tikzpicture}
\end{center}
In proving our Main Lemma, the first part of B{\l}aszczyk and Szyma\'nski's proof goes through: every $G$-flow of weight $\aleph_1$ is a length-$\w_1$ inverse limit of metrizable $G$-flows. However, we run into trouble with the analogue of Lemma~\ref{lem:easylift}: the analogous lemma for flows is false (see Example 3.4 in \cite{Brian}).

The key to getting around this problem is to build our inverse limit system using metrizable reflections. Specifically, suppose $\psi$ is a weight-$\aleph_1$ flow. If $\seq{M_\a}{\a < \w_1}$ is a continuous increasing chain of countable elementary submodels of $H$, then it gives rise to a sequence $\seq{\psi^{M_\a}}{\a < \w_1}$ of metrizable flows, and it is clear that
$\psi = \invlim_{\a < \w_1}\psi^{M_\a}$.
Furthermore, the elementarity between the models gives this inverse limit system a strong degree of coherence, and ultimately is the key that unlocks a workable substitute for Lemma~\ref{lem:easylift}.

This technique is by now old news. It was first used by Dow and Hart in \cite{Dow&Hart} to prove a ``connected version'' of Parovi\v{c}enko's theorem:
\begin{itemize}
\item Every continuum of weight $\leq\! \aleph_1$ is a continuous image of $\mathbb H^*$, the Stone-\v{C}ech remainder of the half-line $\mathbb H = [0,\infty)$.
\end{itemize}
Their techniques were first adapted to a dynamical setting by the author in \cite{Brian}, where it was proved that a weight $\leq\!\aleph_1$ dynamical system $(X,f)$ is a quotient of $(\w^*,\s)$ if and only if it is weakly incompressible. (Here $\s$ denotes the shift map, and $(X,f)$ is weakly incompressible if $f(\closure{U}) \not\sub U$ for every open $U$ with $\0 \neq U \neq X$.)
We should mentioned that the results in \cite{Brian} follow almost immediately from the Main Lemma (this is proved in the next section), and that the Dow-Hart theorem follows from the results in \cite{Brian} (a short, easy proof is given in Section 5 of \cite{Brian}). Thus the Main Lemma is stronger than the theorems in \cite{Brian} and \cite{Dow&Hart}, but the main idea of the proof is the same.

Several technical difficulties arise when applying this idea to arbitrary flows. The main achievement of the proof here is simply to overcome these difficulties and show that the technique can be applied in this very general setting. But this achievement is primarily technical and therefore not as exciting as some of the applications that motivated it. We will set these technicalities aside for now and instead turn to the applications: finding universal dynamical systems of various kinds, and classifying the universal automorphisms of $\pwmf$.

\section{Autohomeomorphisms of $\w^*$ (and automorphisms of $\pwmf$)}\label{sec:dynamics}

\subsection{Dynamical systems and Stone duality}

If $S = (\N,+)$, then every continuous $f: X \to X$ generates an $S$-flow, namely $(n,x) \mapsto f^n(x)$. Conversely, every $S$-flow on $X$ is generated in this way. Thus $(\N,+)$-flows are really the same thing as continuous self-maps. Similarly, $(\Z,+)$-flows are really the same thing as self-homeomorphisms.

\begin{definition}
A \emph{dynamical system} is a continuous function from a compact Hausdorff space to itself.
\end{definition}

The terminology we have employed for flows carries over naturally to dynamical systems, sometimes becoming simpler by virtue of this translation:

\begin{itemize}
\item Given two dynamical systems $f: X \to X$ and $g: Y \to Y$, $f$ is a \emph{quotient} of $g$ if there is a continuous surjection $\pi: Y \to X$ such that $\pi \circ g = f \circ \pi$. This is denoted $g \quotient f$.
\end{itemize}
This is more of a lemma than a definition: our claim is that $g \quotient f$ if and only if the corresponding $(\N,+)$-flow generated by $f$ is a quotient of the corresponding $(\N,+)$-flow generated by $g$. This is trivial to verify: it amounts to checking that if $\pi \circ g = f \circ \pi$ then $\pi \circ g^n = f^n \circ \pi$ for all $n$.
\begin{itemize}
\item Let $p: D \to D$ and let $\phi$ represent the $(\N,+)$-action on $D$ generated by $p$, namely $\phi_n = p^n$. If $f: X \to X$ is a metrizable dynamical system, then a sequence $\seq{x_n}{n \in D}$ of points in $X$ is \emph{$p$-like} if and only if it is $\phi$-like with respect to the flow generated by $f$.
\end{itemize}
Again, this is more of a lemma than a definition. It is easy to verify using the uniform continuity of $f$, and we omit the proof.

We are mainly interested in a specific kind of dynamical system: homeomorphisms from $\w^*$ to itself. Via Stone duality, results about autohomeomorphisms of $\w^*$ translate into results about automorphisms of $\pwmf$. In this section and the next we investigate some consequences of Theorem~\ref{thm:main} for autohomeomorphisms of $\w^*$ and automorphisms of $\pwmf$. We assume a basic familiarity with Stone duality, and refer the reader to \cite{johnstone} for more details.

A function $p: D \to D$ is a \emph{mod-finite permutation} of $D$ if there are cofinite subsets $A,B \sub D$ such that $p$ restricts to a bijection $A \to B$. In this case, $p$ induces an automorphism of $\mathcal{P}(D)/\mathrm{fin} \iso \pwmf$, which we denote $p\up$. Equivalenty, $p\up$ is the map defined by setting $p\up([A]) = [p(A)]$ for every $A \sub D$. This function is called the \emph{lifting} of $p$, and any automorphism of $\pwmf$ obtained in this way is called \emph{trivial}. Trivial automorphisms are dual to trivial autohomeomorphisms, in the sense that
\begin{itemize}
\item[$\circ$] An automorphism $\varphi$ of $\pwmf$ is trivial if and only if the dual autohomeomorphism $\varphi^{\mathrm{St}}$ is trivial.
\item[$\circ$] Moreover, if $p$ is a mod-finite permutation of $\w$ then $(p\up)^{\mathrm{St}} = p^*$ and $(p^*)^{\mathrm{St}} = p\up$.
\end{itemize}

Given two automorphisms $\g,\varphi$ of $\pwmf$, $\varphi$ embeds in $\g$ if there is a subalgebra $\AA$ of $\pwmf$ and an isomorphism $e: \pwmf \to \AA$ such that $e \circ \varphi = \g \circ e$. This is denoted $\varphi \embeds \g$. Embeddings of automorphisms are dual to quotients of autohomeomorphisms, in the sense that
\begin{itemize}
\item[$\circ$] If $f$ and $g$ are autohomeomorphisms of $\w^*$, then $g \quotient f$ if and only if $f^{\mathrm{St}} \embeds g^{\mathrm{St}}$. Conversely, if $\varphi$ and $\g$ are automorphisms of $\pwmf$, then $\varphi \embeds \g$ if and only if $\g^{\mathrm{St}} \quotient \varphi^{\mathrm{St}}$.
\item[$\circ$] Moreover, if $q$ is a quotient mapping from $g$ to $f$ then $q^{\mathrm{St}}$ is an embedding from $g^{\mathrm{St}}$ to $f^{\mathrm{St}}$, and if $e$ is an embedding from $\varphi$ to $\g$ then $e^{\mathrm{St}}$ is a quotient mapping from $\g^{\mathrm{St}}$ to $\varphi^{\mathrm{St}}$.
\end{itemize}

\vspace{1mm}
\begin{center}
\begin{tikzpicture}

\node at (0,0) {$\w^*$};
\node at (2,0) {$\w^*$};
\node at (2,2) {$\w^*$};
\node at (0,2) {$\w^*$};
\draw[->] (.35,2) -- (1.65,2); \node at (1,2.25) {$g$};
\draw[->] (.35,0) -- (1.65,0); \node at (1,-.25) {$f$};
\draw[->>] (2,1.68) -- (2,.3); \node at (2.2,1) {\small $q$};
\draw[->>] (0,1.68) -- (0,.3); \node at (-.2,1) {\small $q$};
\node at (1,-1.25) {\Large $g \quotient f$};

\node at (6,0) {$\pwmf$};
\node at (9,0) {$\pwmf$};
\node at (9,2) {$\pwmf$};
\node at (6,2) {$\pwmf$};
\draw[->] (8.05,2) -- (6.95,2); \node at (7.5,2.25) {$\g$};
\draw[->] (8.05,0) -- (6.95,0); \node at (7.5,-.25) {$\varphi$};
\draw[right hook->] (9,.35) -- (9,1.68); \node at (9.2,1) {\small $e$};
\draw[right hook->] (6,.35) -- (6,1.68); \node at (5.8,1) {\small $e$};
\node at (7.5,-1.25) {\Large $\varphi \embeds \g$};

\draw[double,double equal sign distance,-implies] (3.6,1) -- (4.5,1);
\draw[double,double equal sign distance,-implies] (4.4,1) -- (3.5,1);

\end{tikzpicture}
\end{center}

 \subsection{Universal dynamical systems and automorphisms}
 
 An automorphism of $\pwmf$ is \emph{universal} if every other automorphism embeds in it. We will call an automorphism universal with respect to some property $P$ if it has property $P$, and every other automorphism with property $P$ embeds in it.

Let $t$ be the permutation of $\w \times \Z$ given by $t(n,z) = (n,z+1)$.

\vspace{2mm}\begin{center}
\begin{tikzpicture}[style=thin]

\draw (-3,-2) node {\large$t$};
\draw (-1.2,-1.5) node {$\dots$};
\draw[fill=black] (0,-1.5) circle (1.5pt);
\draw[fill=black] (1,-1.5) circle (1.5pt);
\draw[fill=black] (2,-1.5) circle (1.5pt);
\draw[fill=black] (3,-1.5) circle (1.5pt);
\draw[fill=black] (4,-1.5) circle (1.5pt);
\draw (5.2,-1.5) node {$\dots$};
\draw[->] (-.5,-1.5) -- (-.15,-1.5);
\draw[->] (.15,-1.5) -- (.85,-1.5);
\draw[->] (1.15,-1.5) -- (1.85,-1.5);
\draw[->] (2.15,-1.5) -- (2.85,-1.5);
\draw[->] (3.15,-1.5) -- (3.85,-1.5);
\draw[->] (4.15,-1.5) -- (4.5,-1.5);
\draw (-1.2,-2.2) node {$\dots$};
\draw[fill=black] (0,-2.2) circle (1.5pt);
\draw[fill=black] (1,-2.2) circle (1.5pt);
\draw[fill=black] (2,-2.2) circle (1.5pt);
\draw[fill=black] (3,-2.2) circle (1.5pt);
\draw[fill=black] (4,-2.2) circle (1.5pt);
\draw (5.2,-2.2) node {$\dots$};
\draw[->] (-.5,-2.2) -- (-.15,-2.2);
\draw[->] (.15,-2.2) -- (.85,-2.2);
\draw[->] (1.15,-2.2) -- (1.85,-2.2);
\draw[->] (2.15,-2.2) -- (2.85,-2.2);
\draw[->] (3.15,-2.2) -- (3.85,-2.2);
\draw[->] (4.15,-2.2) -- (4.5,-2.2);
\draw (-1.2,-2.9) node {$\dots$};
\draw[fill=black] (0,-2.9) circle (1.5pt);
\draw[fill=black] (1,-2.9) circle (1.5pt);
\draw[fill=black] (2,-2.9) circle (1.5pt);
\draw[fill=black] (3,-2.9) circle (1.5pt);
\draw[fill=black] (4,-2.9) circle (1.5pt);
\draw (5.2,-2.9) node {$\dots$};
\draw[->] (-.5,-2.9) -- (-.15,-2.9);
\draw[->] (.15,-2.9) -- (.85,-2.9);
\draw[->] (1.15,-2.9) -- (1.85,-2.9);
\draw[->] (2.15,-2.9) -- (2.85,-2.9);
\draw[->] (3.15,-2.9) -- (3.85,-2.9);
\draw[->] (4.15,-2.9) -- (4.5,-2.9);
\draw (.3,-3.5) node {$\vdots$};
\draw (2,-3.5) node {$\vdots$};
\draw (3.7,-3.5) node {$\vdots$};

\end{tikzpicture}\end{center}\vspace{2mm}

\begin{theoremdual}\label{thm:main2}
Assuming \ch,
\begin{enumerate}
\item $t^*$ is universal for bijective dynamical systems of weight $\leq\!\continuum$.
\item $t\up$ is a universal automorphism.
\end{enumerate}
\end{theoremdual}
\begin{proof}
$(1)$ follows directly from Theorem~\ref{thm:main1} (Main Theorem 1) by setting $G = (\Z,+)$.
$(2)$ follows from $(1)$ via Stone duality.
\end{proof}

Part $(1)$ of this theorem can be strengthened as follows:

\begin{theorem}\label{thm:surjections}
Assuming \ch, $t^*$ is universal for surjective dynamical systems of weight $\leq\!\continuum$.
\end{theorem}
\begin{proof}
By Theorem~\ref{thm:main}, it suffices to prove that if $X$ is a compact metric space and $f: X \to X$ a continuous surjection, then $X$ contains a tail-dense $t$-like sequence of points.

Suppose $f: X \to X$ is a continuous surjection, with $X$ metrizable.
Let $\seq{d_n}{n \in \w}$ be a tail-dense sequence of points in $X$. For all $n \in \N$ and $z \geq 0$, define $x_{n,z} = f^z(d_n)$. For negative $z$, where $f^z(d_n)$ may consist of many points, we define $x_{n,z}$ by recursion. For fixed $n$, use backwards recursion on $z$ to choose a sequence $x_{n,-1},x_{n,-2},x_{n,-3},\dots$ of points in $X$ such that $f(x_{n,z}) = x_{n,z+1}$ for all $z < 0$. The sequence $\seq{x_{n,z}}{n \in \w, z \in \Z}$ defined in this way is clearly $t$-like. Because each $d_n$ appears in the sequence, it is tail-dense. 
\end{proof}

If we wish for something universal for all dynamical systems, not just the surjective ones, then $t^*$ will not work, because every quotient of a surjective function is surjective.
Let $u$ be the map on $\w \times \w$ given by $(m,n) \mapsto (m,n+1)$. 

\vspace{2mm}
\begin{center}
\begin{tikzpicture}[style=thin]

\draw (-2,-5.5) node {\large$u$};
\draw[fill=black] (0,-5) circle (1.5pt);
\draw[fill=black] (1,-5) circle (1.5pt);
\draw[fill=black] (2,-5) circle (1.5pt);
\draw[fill=black] (3,-5) circle (1.5pt);
\draw[fill=black] (4,-5) circle (1.5pt);
\draw (5.2,-5) node {$\dots$};
\draw[->] (.15,-5) -- (.85,-5);
\draw[->] (1.15,-5) -- (1.85,-5);
\draw[->] (2.15,-5) -- (2.85,-5);
\draw[->] (3.15,-5) -- (3.85,-5);
\draw[->] (4.15,-5) -- (4.5,-5);
\draw[fill=black] (0,-5.7) circle (1.5pt);
\draw[fill=black] (1,-5.7) circle (1.5pt);
\draw[fill=black] (2,-5.7) circle (1.5pt);
\draw[fill=black] (3,-5.7) circle (1.5pt);
\draw[fill=black] (4,-5.7) circle (1.5pt);
\draw (5.2,-5.7) node {$\dots$};
\draw[->] (.15,-5.7) -- (.85,-5.7);
\draw[->] (1.15,-5.7) -- (1.85,-5.7);
\draw[->] (2.15,-5.7) -- (2.85,-5.7);
\draw[->] (3.15,-5.7) -- (3.85,-5.7);
\draw[->] (4.15,-5.7) -- (4.5,-5.7);
\draw[fill=black] (0,-6.4) circle (1.5pt);
\draw[fill=black] (1,-6.4) circle (1.5pt);
\draw[fill=black] (2,-6.4) circle (1.5pt);
\draw[fill=black] (3,-6.4) circle (1.5pt);
\draw[fill=black] (4,-6.4) circle (1.5pt);
\draw[->] (.15,-6.4) -- (.85,-6.4);
\draw[->] (1.15,-6.4) -- (1.85,-6.4);
\draw[->] (2.15,-6.4) -- (2.85,-6.4);
\draw[->] (3.15,-6.4) -- (3.85,-6.4);
\draw[->] (4.15,-6.4) -- (4.5,-6.4);
\draw (5.2,-6.4) node {$\dots$};
\draw (.3,-7) node {$\vdots$};
\draw (2,-7) node {$\vdots$};
\draw (3.7,-7) node {$\vdots$};

\end{tikzpicture}\end{center}%\vspace{2mm}

\begin{theorem}
Assuming \ch, $u^*$ is universal for dynamical systems of weight $\leq\!\continuum$.
\end{theorem}
\begin{proof}
This follows directly from Theorem~\ref{thm:main1} by setting $G = (\N,+)$ (see the comments following the proof, where it is pointed out that the theorem holds for $(\N,+)$, despite its not being a group).
\end{proof}

Turning back to automorphisms of $\pwmf$, it is natural to ask whether $t^\uparrow$ is the only one that is universal. In some sense of the word ``only'' the answer is no, as the following theorem shows.

\begin{theorem}\label{thm:universals}
Assuming \ch, there are $2^\continuum$ universal automorphisms of $\pwmf$.
\end{theorem}
\begin{proof}
By Theorem 1.2.6 in \cite{JvM}, \ch implies that $\w^*$ is homeomorphic to $(\w \times \Z \times 2^\continuum)^*$, where $2^\continuum$ denotes the weight-$\continuum$ Cantor cube. We may view $2^\continuum$ as a topological group. For each $g \in 2^\continuum$, the map $\varphi_g$ defined by
$$\varphi_g(n,z,x) = (n,z+1,gx)$$
induces an autohomeomorphism $\varphi_g^*$ of $(\w \times \Z \times 2^\continuum)^*$, and it is easy to see that distinct $g$ give rise to distinct autohomeomorphisms. For each $g \in 2^\continuum$, the map $q_g: \w \times \Z \times 2^\continuum \to \w \times \Z$ defined by $q_g(n,z,x) = (n,z)$ induces a continuous surjection $q_g^*: (\w \times \Z \times 2^\continuum)^* \to (\w \times \Z)^*$. For each $g \in 2^\continuum$, we have $q_g \circ \varphi_g = t \circ q_g$, and it follows that $q_g^* \circ \varphi_g^* = t^* \circ q_g^*$. Thus $\varphi_g^* \quotient t^*$ for all $g \in 2^\continuum$. But the $\quotient$ relation is clearly transitive, so by Theorem~\ref{thm:main2}, each $\varphi_g^*$ is universal for bijective dynamical systems of weight $\leq\!\continuum$.
By Stone duality, $\varphi_g^\uparrow$ is a universal automorphism of $\pwmf$ for all $g \in 2^\continuum$.
\end{proof}

Despite this result, one still wonders whether any two universal automorphisms are ``really'' different. Let us make this notion precise:

\begin{definition}$\ $
\begin{itemize}
\item Two autohomeomorphisms $f$ and $g$ of a space $X$ are \emph{isomorphic} if there is a third autohomeomorphism $h$ such that $h \circ f = g \circ h$.
\item Dually, two automorphisms $\varphi$ and $\g$ of $\pwmf$ are \emph{isomorphic} if there is a third automorphism $\eta$ such that $\eta \circ \g = \varphi \circ \eta$.
\end{itemize}
\end{definition}

\emph{A priori}, the isomorphism class of an automorphism might be as large as $2^\continuum$. Thus it is not clear from Theorem~\ref{thm:universals} whether there are any non-isomorphic universal automorphisms. Later in this section, we will characterize the universal trivial automorphisms precisely (under \ch). Then we will see that even among the trivial automorphisms, there are non-isomorphic universal maps.

Of course, there are many automorphisms that are not universal. The identity map is the most obvious example: it is anti-universal, in the sense that no other automorphism embeds in it. After characterizing the universal trivial automorphisms (under \ch), we will see that ``most'' of them (in the sense of Baire category) are not universal. On our way to this classification, we must analyze two other critical mod-finite permutations of $\w$.

\subsection{Universal automorphisms with additional properties}

In this section we analyze two other mod-finite permutations of $\w$, and show that their liftings are universal for automorphisms with certain properties.

Let $r$ denote the permutation of $\w$ consisting of infinitely many disjoint finite cycles, one cycle of size $n!$ for each $n \in \N$. Let $s$ denote the successor map $s(n) = n+1$. The maps $s^*$ and $s^\uparrow$ are both called the \emph{shift map}.

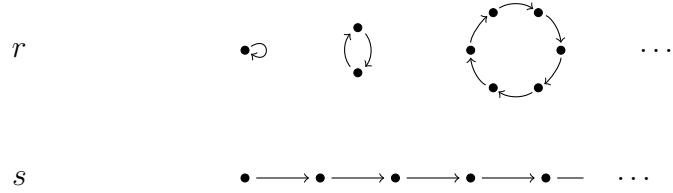
\vspace{2mm}\begin{center}
\begin{tikzpicture}[style=thin]

\draw (-3,2) node {\large$r$};
\draw[fill=black] (0,2) circle (1.5pt);
\draw[fill=black] (1.5,2.3) circle (1.5pt);
\draw[fill=black] (1.5,1.7) circle (1.5pt);
\draw[fill=black] (3,2) circle (1.5pt);
\draw[fill=black] (3.3,2.5) circle (1.5pt);
\draw[fill=black] (3.3,1.5) circle (1.5pt);
\draw[fill=black] (3.9,2.5) circle (1.5pt);
\draw[fill=black] (3.9,1.5) circle (1.5pt);
\draw[fill=black] (4.2,2) circle (1.5pt);
\draw (5.5,2) node {$\dots$};
\draw[->] (.08,2.05) .. controls (.35,2.22) and (.35,1.78) .. (.08,1.95);
\draw[->] (1.6,2.22) .. controls (1.7,2.1) and (1.7,1.9) .. (1.6,1.78);
\draw[->] (1.4,1.78) .. controls (1.3,1.9) and (1.3,2.1) .. (1.4,2.22);
\draw[->] (3,2.1) .. controls (3,2.2) and (3.15,2.4) .. (3.22,2.45);
\draw[->] (3.38,2.56) .. controls (3.5,2.64) and (3.7,2.64) .. (3.83,2.55);
\draw[->] (4,2.46) .. controls (4.05,2.45) and (4.18,2.3) .. (4.2,2.1);
\draw[->] (3.2,1.54) .. controls (3.15,1.55) and (3.02,1.7) .. (3,1.9);
\draw[->] (3.82,1.44) .. controls (3.7,1.36) and (3.5,1.36) .. (3.37,1.45);
\draw[->] (4.2,1.9) .. controls (4.2,1.8) and (4.05,1.6) .. (3.98,1.55);

\draw (-3,.3) node {\large$s$};
\draw[fill=black] (0,.3) circle (1.5pt);
\draw[fill=black] (1,.3) circle (1.5pt);
\draw[fill=black] (2,.3) circle (1.5pt);
\draw[fill=black] (3,.3) circle (1.5pt);
\draw[fill=black] (4,.3) circle (1.5pt);
\draw (5.2,.3) node {$\dots$};
\draw[->] (.15,.3) -- (.85,.3);
\draw[->] (1.15,.3) -- (1.85,.3);
\draw[->] (2.15,.3) -- (2.85,.3);
\draw[->] (3.15,.3) -- (3.85,.3);
\draw (4.15,.3) -- (4.5,.3);

\end{tikzpicture}\end{center}\vspace{2mm}

We adopt the convention that ``$A \subset B$'' means $A$ is a strict subset of $B$.

\begin{definition} Let $X$ be a zero-dimensional compact Hausdorff space, and let $f: X \to X$ be a dynamical system.
\begin{itemize}
\item $f$ is \emph{chain transitive} if for every clopen $U \sub X$ with $\0 \neq U \neq X$, $f(U) \not\sub U$.
\item $f$ is \emph{chain recurrent} if for every clopen $U \sub X$, $f(U) \not\subset U$.
\end{itemize}
Dually, if $\varphi$ is an automorphism of $\pwmf$ then
\begin{itemize}
\item $\varphi$ is \emph{chain transitive} if for every $a \in \pwmf$ with $[\0] \neq a \neq [\w]$, $\varphi(a) \not\leq a$.
\item $\varphi$ is \emph{chain recurrent} if for every $a \in \pwmf$, $\varphi(a) \not< a$.
\end{itemize}
\end{definition}

For now, we have stated the definitions of chain transitive/recurrent dynamical systems for zero-dimensional spaces only. Soon we will extend the definition to all dynamical systems, but first let us state the theorem to be proved concerning $s$ and $r$.

\begin{theorem}\label{thm:dynamics}
Assuming \ch,
\begin{enumerate}
\item $s^*$ is universal for chain transitive dynamical systems of weight $\leq\!\continuum$.
\item $r^*$ is universal for chain recurrent dynamical systems of weight $\leq\!\continuum$.
\end{enumerate}
By Stone duality, it follows that
\begin{enumerate}
\item $s^\uparrow$ is universal for chain transitive automorphisms of $\pwmf$.
\item $r^\uparrow$ is universal for chain recurrent automorphisms of $\pwmf$.
\end{enumerate}
\end{theorem}

Theorem~\ref{thm:dynamics}$(1)$ is the main result of \cite{Brian}. The (straightforward) proof in this section therefore shows that the results in \cite{Brian} are a special case of Theorem~\ref{thm:main}.

We will prove Theorem~\ref{thm:dynamics} through a sequence of lemmas.

\begin{lemma}\label{lem:ct1}
$\ $
\begin{enumerate}
\item $s^*$ is chain transitive.
\item $r^*$ is chain recurrent. Moreover, if $\tilde r\!$ is any permutation of $\w$ consisting only of finite cycles, then $\tilde r^*\!$ is chain recurrent.
\end{enumerate}
\end{lemma}
\begin{proof}
For $(2)$, it is easy to see that $\tilde r(A) \not\subset A$ for any $A \sub \w$. It follows that $\tilde r^*(A^*) \not\subset A^*$ for any $A \sub \w$, so that $\tilde r^*\!$ is chain recurrent. $(1)$ is proved similarly; it also appears as Lemma 5.3 in \cite{WRB}. 
\end{proof}

\begin{lemma}\label{lem:ct2}
Both chain transitivity and chain recurrence are preserved by taking quotients.
\end{lemma}
\begin{proof}
This is proved in chapter 4 of \cite{akin}. The result is stated there for metrizable dynamical systems only, but the proof does not use this.
\end{proof}

Observe that these two lemmas suffice already for the ``only if'' direction of Theorem~\ref{thm:dynamics}. For the ``if'' direction, we will need alternative characterizations of chain transitivity and recurrence. Let us expand our ``$\approx_\e$'' notation for metric spaces to arbitrary spaces as follows. If $X$ is a topological space and $\O$ is an open cover of $X$, we write $x \approx_\O y$ to mean that $x,y \in U$ for some $U \in \O$. Notice that for a metric space, ``$\approx_\e$'' coincides with ``$\approx_\O$'' when $\O$ is the open cover consisting of all open sets of diameter $\leq \! \e$.

\begin{definition}
Let $f: X \to X$ be a dynamical system.
If $\O$ is an open cover of $X$, then an $\O$\textit{-chain from $a$ to $b$} is a sequence $\langle x_i : i \leq n \rangle$ of points in $X$ with $x_0 = a$, $x_n = b$, and $n \geq 1$, such that $f(x_i) \approx_{\O} x_{i+1}$ for every $i < n$. If $X$ is metrizable, an \emph{$\e$-chain} is defined to be an $\O$-chain where $\O$ is the open cover consisting of all open sets of diameter $\leq \! \e$
\end{definition}

Roughly, an $\O$-chain is an orbit that is computed with small errors at each step, the size of the errors being restricted by $\O$. The following lemma is well-known, at least for metrizable dynamical systems, and a proof can be found in chapter 4 of \cite{akin} (the proofs given there are for metrizable systems, but the reader can check that metrizability is never actually used; see also Lemma 5.2 in \cite{WRB}). 

\begin{lemma}\label{lem:chains}
Let $f: X \to X$ be a zero-dimensional dynamical system.
\begin{enumerate}
\item $f$ is {chain transitive} if and only if for every $a,b \in X$ and every open cover $\O$ of $X$, there is an $\O$-chain from $a$ to $b$.
\item $f$ is {chain recurrent} if and only if for every $a \in X$ and every open cover $\O$ of $X$, there is an $\O$-chain from $a$ to $a$.
\end{enumerate}
\end{lemma}

This lemma explains the origin of the terms ``chain transitive'' and ``chain recurrent''. We also take parts $(1)$ and $(2)$ of this lemma as the definition of chain transitivity and chain recurrence, respectively, in the case that $X$ is not zero-dimensional.

\begin{lemma}\label{lem:chainz}
Let $X$ be a metric space and $f: X \to X$ a dynamical system.
\begin{enumerate}
\item $f$ is chain transitive if and only if $X$ contains a tail-dense $s$-like sequence.
\item $f$ is chain recurrent if and only if $X$ contains a tail-dense $r$-like sequence.
\end{enumerate}
\end{lemma}
\begin{proof}
For $(1)$, the forward direction is proved by Bowen in \cite{bowen}. The idea is to fix a tail-dense sequence $\seq{d_n}{n \in \w}$ of points in $X$, and then to expand this sequence by connecting $d_n$ to $d_{n+1}$ with a $\nicefrac{1}{n}$-chain.

For the converse direction, suppose $(X,f)$ is a metrizable dynamical system and that $\seq{x_n}{n \in \w}$ is a tail-dense $s$-like sequence of points in $X$. Let $\e > 0$, and fix $N \in \N$ such that $f(x_n) \approx_\e x_{n+1}$ for all $n \geq N$. Let $a,b \in X$. Because every tail of $\seq{x_n}{n \in \w}$ is dense, there is some $m \geq N$ and some $n \geq m$ such that $f(a) \approx_\e x_m$ and $b \approx_\e x_n$. Then
$$\<a,x_m,x_{m+1},\dots,x_{n-2},x_{n-1},b\>$$
is an $\e$-chain from $a$ to $b$. 

The proof of $(2)$ is similar, but slightly more invovled. For the forward direction, let $f: X \to X$ be a chain recurrent dynamical system, with $X$ metrizable, and let $\seq{d_k}{k \in \w}$ be a tail-dense sequence of points in $X$. For every $k \in \N$, fix $n_k \in \N$ such that there is a $\nicefrac{1}{k}$-chain
$$\<d_k,x^k_1,x^k_2,\dots,x^k_{n_k-2},x^k_{n_k-1},d_k\>$$
of length $n_k+1$ from $d_k$ to itself (and set $x^k_0 = d_k$ for convenience). Furthermore, suppose that the function $k \mapsto n_k$ is strictly increasing. This assumption sacrifices no generality, because we can always increase the length of our chain by repeating it if necessary:
$$\<d_k,x^k_1,x^k_2,\dots,x^k_{n_k-2},x^k_{n_k-1},d_k,x^k_1,x^k_2,\dots,x^k_{n_k-2},x^k_{n_k-1},d_k\>.$$

For convenience, we take $D = \bigcup_{n \in \N}\{n\} \times n!$ to be the domain of $r$, with $r(n,m) = (n,m+1)$ and with the addition understood modulo $n!$.
For $n \in \N$ and $m < n!$, define
\[
x_{n,m} = 
\begin{cases} 
\text{any point} & \text{ if } n < n_1, \\
x^k_{m\,(\text{mod }n_k)} & \text{ if }n_k \leq n < n_{k+1} \text{ and } 0 \leq m < n!.
\end{cases}
\]
In other words, we define our $r$-like sequence by mapping the cycle of length $n!$ onto the $\nicefrac{1}{k}$-chain from $d_k$ to itself whenever $n_k \leq n < n_{k+1}$.

Notice that if $n_k \leq n < n_{k+1}$, then $n_k$ divides $n!$, so that
$$\<x_{n,0},x_{n,1},x_{n,2},\dots,x_{n,n!-2},x_{n,n!-1},x_{n,0}\>$$
is a $\nicefrac{1}{k}$-chain from $d_k$ to itself; it is just the $\nicefrac{1}{k}$-chain we began with, repeated $\nicefrac{n!}{n_k}$ times. 
If $\e > 0$, fix $k$ such that $\nicefrac{1}{k} < \e$; then $f(x_{n,m}) \approx_\e x_{n,m+1}$ for all $n \geq n_k$. Thus the sequence $\seq{x_{n,m}}{n \in \N,m < n!}$ is $r$-like. It is clear that the sequence is also tail-dense, because it contains all the $d_k$. 

For the converse direction, suppose $f: X \to X$ is a dynamical system and $\seq{x_{n,m}}{n \in \N, m < n!}$ is a tail-dense $r$-like sequence of points in $X$. Let $\e > 0$ and let $a \in X$; we must find an $\e$-chain from $a$ to itself. Let $\dlt > 0$ be small enough that
\begin{itemize}
\item if $x \approx_{\dlt} y \approx_{\dlt} z$, then $x \approx_{\e} z$, and
\item if $x \approx_{\dlt} y$ and $f(y) \approx_{\dlt} z$, then $f(x) \approx_{\e} z$.
\end{itemize} 
Fix $N \in \N$ such that $f(x_{n,m}) \approx_\dlt x_{n,m+1}$ for every $n \geq N$, where the addition is understood modulo $n!$. Because every our sequence is tail-dense, there is some $n \geq N$ and $m < n!$ such that $a \approx_{\dlt} x_{n,m}$. It follows from our choice of $\dlt$ that
$$\<a,x_{n,m+1},x_{n,m+1}, \dots , x_{n,m+n!-1}, a\>$$
is an $\e$-chain from $a$ to $a$.
\end{proof}

\begin{proof}[Proof of Theorem~\ref{thm:dynamics}]
For $(1)$, the ``only if'' part follows from Lemmas \ref{lem:ct1}$(1)$ and \ref{lem:ct2}$(1)$. For the ``if'' part, Lemma~\ref{lem:chainz}$(1)$ shows that every metrizable chain transitive dynamical system contains a tail-dense $s$-like sequence. If $(X,f)$ is chain transitive then so is each of its metrizable quotients by Lemma~\ref{lem:ct1}$(1)$, so the theorem follows from Theorem~\ref{thm:main}. $(2)$ is proved the same way.
\end{proof}

One may easily check that the proof of Theorem~\ref{thm:dynamics} goes through, with only minor modifications, for $s^{-1}$ instead of $s$. Thus $s\up$ and $(s\up)^{-1}$ both are universal chain transitive automorphisms of $\pwmf$ (under \ch). No other trivial maps are chain transitive, up to re-indexing the domain of $s$ or $s^{-1}$ (this is Lemma 5.5 in \cite{WRB}).

Let us also observe that the proof of Theorem~\ref{thm:dynamics} still works if $r$ is replaced with any permutation $\tilde r\!$ such that
\begin{itemize}
\item $\tilde r\!$ decomposes into a union of finite cycles, and
\item for each $k \in \N$ all but finitely many of periods of these cycles are divisible by $k$.
\end{itemize}

\subsection{Every trivial automorphism embeds in its inverse}

In this subsection, we will break briefly from classifying the universal trivial automorphisms to observe an interesting consequence of Theorem~\ref{thm:dynamics}$(2)$.

\begin{corollary}\label{cor:inverses}
Assuming \ch, $(s^*)^{-1}$ is a quotient of $s^*$.
\end{corollary}
\begin{proof}[Proof]
It is enough to show that $(s^*)^{-1}$ is chain transitive. But we already know that $s^*$ is chain transitive, and in general an invertible map is chain transitive if and only if its inverse is (any two points can still be connected by a chain -- the chains just run in the opposite direction).
\end{proof}

This corollary was first observed in Section 5 of \cite{Brian}, where it was also proved that the assumption of \ch cannot be dropped: under \ocama, the shift map is not a quotient of its inverse (see Theorem 5.7 in \cite{Brian}).

Corollary~\ref{cor:inverses} constitutes some small progress on what seems to be a difficult open question:

\begin{question}\label{q:shiftmap}
Is it consistent that the shift map $s^\uparrow$ is isomorphic to its inverse? Does it follow from \ch?
\end{question}

For more partial progress on this question, see \cite{geschke}. In this subsection, we will see that Corollary~\ref{cor:inverses} extends to every trivial map under \ch, and that the answer to Question~\ref{q:shiftmap} may tell us something about all of them.

\begin{theorem}\label{thm:inverses}
Assuming \ch,
\begin{enumerate}
\item every trivial autohomeomorphism of $\w^*$ is a quotient of its inverse.
\item every trivial automorphism of $\pwmf$ embeds in its inverse.
\end{enumerate}
Furthermore, if \ch holds and the answer to Question~\ref{q:shiftmap} is positive, then
\begin{enumerate}
\item every trivial autohomeomorphism of $\w^*$ is isomorphic to its inverse.
\item every trivial automorphism of $\pwmf$ is isomorphic to its inverse.
\end{enumerate}
\end{theorem}

Before proving this theorem, we will require a bit more terminology.

\begin{definition}
Let $p: \w \to \w$ be a mod-finite permutation of $\w$. $A \sub \w$ is \emph{fixed by $p$} if for all $n$ in the domain of $p$, $p(n) \in A$ if and only if $n \in A$.
\begin{itemize}
\item $A \sub \w$ is a \emph{$\Z$-orbit} if $A$ is fixed by $p$ and if $A = \set{a_z}{z \in \Z}$, where
\begin{itemize}
\item[$\circ$] $a_{n+1} = p(a_n)$ for all $n \in \N$,
\item[$\circ$] all the $a_n$ are distinct.
\end{itemize}
\item $A \sub \w$ is an \emph{$\N$-orbit} if $A$ is fixed by $p$ and if $A = \set{a_n}{n \in \N}$, where
\begin{itemize}
\item[$\circ$] $a_{n+1} = p(a_n)$ for all $n \in \N$,
\item[$\circ$] all the $a_n$ are distinct.
\end{itemize}
\item $A \sub \w$ is a \emph{backwards $\N$-orbit} if it is an $\N$-orbit for the mod-finite permutation $p^{-1}$.
\item $A \sub \w$ is an \emph{$n$-cycle} if $|A| = n$ and $p$ cyclically permutes the elements of $A$. $A$ is called a \emph{finite cycle} if it is an $n$-cycle for some $n \in \N$, and in this case $n$ is called the \emph{period} of the cycle.
\end{itemize}
\end{definition}

%Every permutation of $\w$ induces a partition of $\w$ into the four orbit types described above. This is not true for mod-finite permutations generally, but it is ``mod-finite true'' in the following sense:

\begin{lemma}
For every mod-finite permutation $p$ of $\w$, there is a mod-finite permutation $q$ of $\w$ with $p^* = q^*$ and $p^\uparrow = q^\uparrow$, such that $q$ induces a partition of $\w$ into the four types of orbits described above. Furthermore, only finitely many members of this partition are either $\N$-orbits or backwards $\N$-orbits, and only one of these two types is included at all.
\end{lemma}

To prove the first assertion of this lemma, it suffices to construct a mod-finite permutation $q$ such that $p$ and $q$ differ in only finitely many places and the domain of $q$ can be partitioned into the four types of orbits described above. This is an easy exercise, and we omit the proof. For the second assertion, it suffices to notice that an $\N$-orbit and a backwards $\N$-orbit can be combined into a $\Z$-orbit (by extending the domain of $p$ by one point). In light of this lemma, we may and do assume from now on that every mod-finite permutation induces a partition of $\w$ into the four types of orbits as described above.

\begin{proof}[Proof of Theorem~\ref{thm:inverses}]
Let $p$ be a mod-finite permutation of $\w$. Assume that $p$ has $n$ $\N$-orbits and no backwards $\N$-orbits (the proof is nearly identical if we assume it has $n$ backwards $\N$-orbits and no $\N$-orbits). Let $A_0,A_1,\dots,A_{n-1}$ denote the $\N$-orbits of $p$, and let $B = \w \setminus \bigcup_{i < n}A_i$. 

Observe that $p \rest B$ is a bijection, and is naturally isomorphic to its inverse (each $\Z$-orbit and each $m$-cycle remains a $\Z$-orbit or an $m$-cycle -- it just runs in the opposite direction under $(p \rest B)^{-1}$). Thus $p^* \rest B^*$ is isomorphic to $(p^*)^{-1} \rest B^*$. By Corollary~\ref{cor:inverses}, $p^* \rest A_i^*$ is a quotient of $(p^*)^{-1} \rest A_i^*$ for each $i < n$. If the answer to Question~\ref{q:shiftmap} is positive, then $p^* \rest A_i^*$ is isomorphic to $(p^*)^{-1} \rest A_i^*$ for each $i < n$.

Pasting these quotient mappings (or isomorphisms) together proves part $(1)$ of the theorem, and part $(2)$ follows via Stone duality.
\end{proof}

It is an open question whether any of the results of this section can be extended to non-trivial automorphisms:

\begin{question}
Assuming \ch, does every automorphism of $\pwmf$ embed in its inverse (even the nontrivial ones)? Is every automorphism isomorphic to its inverse?
\end{question}

\subsection{Classifying the trivial universal maps}

We mentioned already at the end of Section 3.3 that the proof of Theorem~\ref{thm:dynamics} shows something a little stronger than the statement of the theorem. We will use this stronger statement in this subsection:

\begin{lemma}\label{lem:cycles}
Assuming \ch, 
\begin{enumerate}
\item Both $s^*$ and $(s^*)^{-1}$ are universal for chain transitive dynamical systems of weight $\leq\!\continuum$.
\item Let $\tilde r\!$ be a permutation of $\w$ consisting only of finite cycles, and suppose that for every $k \in \N$, all but finitely many of the finite cycles of $\tilde r\!$ have period divisible by $k$. Then
%\begin{enumerate}
%\item 
$\tilde r^*$ is universal for chain recurrent dynamical systems of weight $\leq\!\continuum$.
%\item $p^\uparrow$ is universal for chain recurrent automorphisms of $\pwmf$.
%\end{enumerate}
\end{enumerate}
\end{lemma}

\begin{definition}
If $p$ is a mod-finite permutation of $\w$, let us say that $p$ is \emph{pan-divisible} if for every $k \in \N$, all but finitely many of the finite cycles of $p$ have period divisible by $k$.
\end{definition}

For example, $r$, $s$, and $t$ are all pan-divisible: $r$ is pan-divisible because any fixed $k$ divides $n!$ for large enough $n$, and $s$ and $t$ are pan-divisible vacuously, because they have no finite cycles. 

\begin{theorem}\label{thm:orbitz}
Assume \ch, and let $p$ be a mod-finite permutation of $\w$.
The following are equivalent:
\begin{enumerate}
\item $p$ has infinitely many $\Z$-orbits and is pan-divisible.
\item $p^*$ is a universal for surjective dynamical systems of weight $\leq\continuum$.
\item $p^\uparrow$ is a universal automorphism of $\pwmf$.
\end{enumerate}
\end{theorem}

Observe that one may endow the set of mod-finite permutations of $\w$ with a natural topology that makes it into a Polish space. As a consequence of this theorem, the set of mod-finite permutations of $\w$ that induce universal automorphisms is meager. In this sense, ``most'' mod-finite permutations do not give rise to universal automorphisms.

\begin{proof}[Proof of Theorem~\ref{thm:orbitz}]
By applying Stone duality, it is easy to see that $(2) \Rightarrow (3)$. We will prove the theorem by showing $(1) \Rightarrow (2)$ and $(3) \Rightarrow (1)$.

To prove $(1) \Rightarrow (2)$, suppose $p$ is a mod-finite permutation of $\w$ that has infinitely many $\Z$-orbits and is pan-divisible. Let $A$ denote the union of the $\Z$-orbits of $p$, let $C$ denote the union of the finite cycles of $p$, and let $B = \w \setminus (A \cup C)$ denote the union of the $\N$-orbits or the backwards $\N$-orbits. If $B \cup C$ is finite, then $p^* = t^*$ and the result follows from Theorem~\ref{thm:surjections}. So let us suppose $B \cup C$ is infinite.

Let $f: X \to X$ be a surjective dynamical system of weight $\leq\!\continuum$; we must show $p^* \quotient f$. Roughly, the idea is to use the universality of $p^* \rest A^*$ to find a quotient mapping $p^* \rest A^* \to f$, and then to extend this mapping to all of $\w^*$ in a ``harmless'' way. 

We will use the following well-known fact about dynamical systems:

\begin{lemma}
There is a closed subspace $Y \sub X$ such that $f$ maps $Y$ into itself, and the dynamical system $f \rest Y: Y \to Y$ is chain transitive. 
\end{lemma}
\begin{proof}[Proof of the lemma:]
A closed subset $Y$ of $X$ is called \emph{minimal} if $Y \neq \0$, $f$ maps $Y$ into itself, and for every $y \in Y$ the orbit $\set{f^z(y)}{z \in \Z}$ of $y$ is dense in $Y$. It is well-known that every dynamical system contains minimal closed sets (hint: apply Zorn's lemma to the poset of all nonempty, closed, $f$-invariant subsets of $X$). Any such set clearly suffices.
\end{proof}

By Theorem~\ref{thm:surjections}, $p^* \rest A^*$ is universal for surjective dynamical systems of weight $\leq\!\continuum$. Let $q_A: A^* \to X$ be a quotient mapping from $p^* \rest A^*$ to $f$.
By Theorem~\ref{thm:dynamics}, $p^* \rest B^*$ is universal for chain transitive dynamical systems of weight $\leq\!\continuum$. Let $q_B: B^* \to Y$ be a quotient mapping from $p^* \rest B^*$ to $f \rest Y$.
By Lemma~\ref{lem:cycles}, $p^* \rest C^*$ is universal for chain recurrent dynamical systems of weight $\leq\!\continuum$. Let $q_C: C^* \to Y$ be a quotient mapping from $p^* \rest C^*$ to $f \rest Y$.

Pasting these three maps together, we obtain a map $q = q_A \cup q_B \cup q_C$ defined on all of $\w^*$, and it is clear that $q$ is a quotient mapping from $p^*$ to $f$. This completes the proof of $(1) \Rightarrow (2)$.

To show $(3) \Rightarrow (1)$, we must show two things: 
\begin{enumerate}[(a)]
\item if $p^\uparrow$ is universal, then $p$ has infinitely many $\Z$-orbits, and
\item if $p^\uparrow$ is universal, then $p$ is pan-divisible.
\end{enumerate}

For (a), it is easiest to prove the contrapositive of the Stone dual. That is, we will show that if $p$ has only finitely many $\Z$-orbits then $p^*$ is not a universal autohomeomorphism of $\w^*$.

If $p$ has only finitely many $\Z$-orbits, then (by removing one point in each of the $\Z$ orbits from the domain of $p$) each of them may be decomposed into an $\N$-orbit and a backwards $\N$-orbit. Thus $\w$ can be decomposed into three sets, $A$, $B$, and $C$, such that $A$ consists of finitely many $\N$-orbits, $B$ consists of finitely many backwards $\N$-orbits, and $C$ consists of finite cycles. So $p^* \rest A^*$ is a union of finitely many copies of $s^*$, and it follows from Lemma~\ref{lem:ct1}$(1)$ that $p^* \rest A^*$ is chain recurrent. Similarly, $p^* \rest B^*$ is chain recurrent. Lastly, $p^* \rest C^*$ is chain recurrent as well by Lemma~\ref{lem:ct1}. It follows that $p^*$ is chain recurrent. Moreover, every quotient of $p^*$ is chain recurrent by Lemma~\ref{lem:ct2}. But some autohomeomorphisms of $\w^*$ are not chain recurrent (for example, the map $t^*$), so this shows $p^*$ is not universal.

For $(2)$, we will prove the contrapositive: if $p$ is not pan-divisible, then $p^\uparrow$ is not universal.

Supposing $p$ is not pan-divisible, there is some $n \in \N$ and an infinite $A \sub \w$ such that $p \rest A$ is an infinite union of finite cycles, none of which have period divisible by $n$. Let $c_n$ denote the permutation of $\w$ that is a disjoint union of infinitely many $n$-cycles. We claim $c_n^\uparrow \not\embeds p^\uparrow$.

Aiming for a contradiction, suppose $c_n^\uparrow \embeds p^\uparrow$ and let $\BB$ be a subalgebra of $\pwmf$ such that $c_n^\uparrow$ is isomorphic to $p^\uparrow \rest \BB$. Every member of $\pwmf$ is periodic under $c_n^\uparrow$ with a period dividing $n$. In other words, $(c_n^\uparrow)^n(a) = a$ for every $a \in \pwmf$, and it follows that $(p^\uparrow)^n(b) = b$ for every $b \in \BB$. On the other hand, using our assumption about $p$ and $A$, it is easy to check that if $a < [A]$ then $(p^\uparrow)^n(a) \neq a$. From this it follows that for all $b \in \BB$, either $b = [A]$ or $b \leq [\w \setminus A]$. But then $\BB$ is not a subalgebra of $\pwmf$, and this is the desired contradiction.
\end{proof}

In light of Theorem~\ref{thm:orbitz}, it is now easy to show that not all universal automorphisms are isomorphic. This shows, in particular, that two automorphisms of $\pwmf$ may embed in each other without being isomorphic.

\begin{definition}
If $p$ and $q$ are mod-finite permutations of $\w$, let $p \vee q$ denote the mod-finite permutation of $\w \times 2$ that acts like $p$ on $\w \times \{0\}$ and like $q$ on $\w \times \{1\}$. By reindexing, $p \vee q$ is considered a mod-finite permutation of $\w$.
\end{definition}

\begin{example}
By Theorem~\ref{thm:orbitz}, if \ch holds then both $t^*$ and $(t \vee r)^*$ are universal autohomeomorphisms of $\w^*$. However, these autohomeomorphisms are not isomorphic. To see this, consider the following property of an autohomeomorphism $h$:
\begin{itemize}
\item[$(\dagger)$] $h$ is chain recurrent, but its restriction to a clopen subset never yields a chain transitive dynamical system.
\end{itemize} 
It is fairly easy to check that
\begin{itemize}
\item with $(t \vee r)^\uparrow$, there is an invariant, clopen subset of $\w^*$ with property $(\dagger)$, namely the part that is a copy of $r^*$.
\item with $t^\uparrow$, no invariant, clopen subset of $\w^*$ has property $(\dagger)$.
\end{itemize}
In both cases we leave the details of checking this to the reader. It follows that $(t \vee r)^\uparrow$ and $t^\uparrow$ are not isomorphic.
\end{example}

\subsection{Universal automorphisms under \ocama}

We will end this section with a few observations concerning the structure of the $\embeds$ and $\quotient$ relations under $\ocama$. This will establish the independence of some of the results proved in the earlier parts of this section.

We observed already in subsection 3.4 that the results there fail under \ocama, which implies that the shift map and its inverse are not quotients of each other. Working backwards, we will next show that the results of subsection 3.3 are independent, and after that we will consider subsection 3.2, and whether there might be universal automorphisms under \ocama.

We do not need to apply either $\oca$ or $\ma$ directly for any of the results in this section. Instead, we may content ourselves with applying a consequence of $\ocama$, a general version of which was proved by Farah in \cite{farah}. We quote the result without proof:

\begin{theorem}\label{thm:farah}
Assuming \ocama, if $F: \w^* \to \w^*$ is continuous, then there is some $A \sub \w$ such that the image of $F \rest (\w \setminus A)^*$ is nowhere dense, and $F \rest A^*$ is induced by a finite-to-one function $A \to \w$.
\end{theorem}

Note that a special case of this theorem is that under \ocama, all automorphisms of $\pwmf$ are trivial (this consequence of \ocama was known before Theorem~\ref{thm:farah} and is due to Velickovic \cite{velickovic}; the same result under \pfa is due to Shelah-Stepr\={a}ns \cite{S&S}; consistency was first proved by Shelah \cite{shelah}).

%Given a mod-finite permutation $p$ of $\w$, let us say that $A \sub \w$ is \emph{almost $p$-invariant} if 

%\begin{lemma}\label{lem:farah}
%$(\ocama)$ Let $a$ and $b$ be mod-finite permutations of $\w$, and suppose $Q: \w^* \to \w^*$ is any continuous mapping such that $Q \circ a^* = b^* \circ Q$. Then there is some $A \sub \w$ such that the image of $Q \rest (\w \setminus A)^*$ is nowhere dense, and $Q \rest A^*$ is induced by a finite-to-one function $q: A \to \w$, where $A$ is almost $a$-invariant and $q(A)$ is almost $b$-invariant.
%\end{lemma}
%\begin{proof}
%Suppose $Q: \w^* \to \w^*$ is a quotient mapping from $a^*$ to $b^*$, and use Theorem~\ref{thm:farah} to find $A \sub \w$ such that the image of $F \rest (\w \setminus A)^*$ is nowhere dense, and $F \rest A^*$ is induced by a finite-to-one function $A \to \w$. We must show that $A$ is almost $a$-invariant and $q(A)$ is almost $b$-invariant.
%\end{proof}

\begin{definition}
Let $p$ be a mod-finite permutation of $\w$.
\begin{itemize}
\item The \emph{cyclic part} of $p$ is the union of its finite cycles. If $C$ is the cyclic part of $p$, then $p^* \rest C^*$ is called the cyclic part of $p^*$. 
\item If the cyclic part of $p$ is finite, then $p$ and $p^*$ are called \emph{acyclic}.
\end{itemize}
\end{definition}

\begin{theorem}\label{lem:farah2}
Assuming \ocama, $t^*$ has no cyclilc maps as quotients, and no cyclic map has $s^*$ or $(s^*)^{-1}$ as a quotient. Consequently, $t^*$ is not universal, and $r^*$ is not universal for chain transitive autohomeomorphisms.
\end{theorem}
\begin{proof}
We will prove first that $t^*$ has no cyclic maps as quotients. Suppose $c$ is a cyclic mod-finite permutation of $\w$, and suppose that $Q: (\w \times \Z)^* \rightarrow \w^*$ is a quotient mapping. By Theorem~\ref{thm:farah}, there is some $A \sub \w \times \Z$ such that $Q \rest A^*$ is induced by a finite-to-one function $q: A \to \w$, and the image of $Q \rest (\w \times \Z \setminus A)^*$ is nowhere dense.

Because the image of $Q \rest (\w \times \Z \setminus A)^*$ is nowhere dense, $A$ must be infinite and $q(A)$ must be co-finite: otherwise $Q$ could not be a surjection. 

Because $q(A)$ is co-finite, we may assume (by removing finitely many points from $A$, if necessary) that $q(A)$ is an infinite union of disjoint cycles, say $q(A) = \bigcup_{n \in \w}C_n$.

For each $n \in \w$, fix $a_n \in A$ with $q(a_n) \in C_n$. The set $q^{-1}(C_n)$ is finite, so $t^k(a_n) \notin C_n$ for all sufficiently large $k$. On the other hand, $c^k(q(a_n)) \in C_n$ for all $k$. Thus there is for each $n$ some $k_n \in \N$ such that $q(t^{k_n}(a_n)) \in C_n$ but $q(t^{k_n+1}(a_n)) \notin C_n$. Let $b_n = t^{k_n}(a_n)$ for each $n$, and let $B = \set{b_n}{n \in \N}$.

$B$ is an infinite subset of $A$, and by design, we have $q(t(b)) \neq c(q(b))$ for all $b \in B$. It follows that if $\U \in B^*$ then $q^*(t^*(\U)) \neq c^*(q^*(\U))$. This contradicts the supposition that $Q$ is a quotient mapping and establishes $t^* \not\quotient c^*$.

Next we will prove that no cyclic map has $s^*$ as a quotient. Suppose $c$ is a cyclic mod-finite permutation of $\w$ as before, and suppose that $Q$ is a quotient mapping from $c^*$ to $s^*$. By Theorem~\ref{thm:farah}, there is some $A \sub \w$ such that $Q \rest A^*$ is induced by a finite-to-one function $q: A \to \w$, and the image of $Q \rest (\w \setminus A)^*$ is nowhere dense.

As before, $A$ must be infinite and $q(A)$ must be co-finite, since otherwise $Q$ could not be a surjection.

Thus for all but finitely many $n \in \N$, we may find some $a_n \in A$ such that $q(a_n) = n$. Because $c$ is cyclic, there is some $k > 0$ such that $c^{k}(a_n) = a_n$, although 
$$q(c^{k}(a_n)) = q(a_n) = n \neq n+k = q(a_n)+k = s^{k}(q(a_n)).$$
Thus for each $n$, there is some $k_n \geq 0$ such that $q(c^{k_n}(a_n)) = s^{k_n}(q(b))$ but $q(c^{k_n+1}(a_n)) \neq s^{k_n+1}(q(b))$ (for example, we could take $k_n$ to be the least $k$ satisfying the above inequality, minus one). Let $b_n = c^{k_n}(a_n)$, and observe that $q(c(b_n)) \neq s(q(b_n))$. Let $B = \set{b_n}{n \in \N}$.

$B$ is an infinite subset of $A$, and we have $q(c(b)) \neq s(q(b))$ for all $b \in B$. It follows that if $\U \in B^*$ then $q^*(c^*(\U)) \neq s^*(q^*(\U))$. This contradicts the supposition that $Q$ is a quotient mapping and establishes $c^* \not\quotient s^*$.
\end{proof}

\begin{theorem}
Assuming \ocama,
\begin{enumerate}
\item there is no universal chain transitive automorphism.
\item there is no universal chain recurrent automorphism.
\end{enumerate}
\end{theorem}
\begin{proof}
For $(1)$, it was proved as Lemma 5.5 in \cite{WRB} that $s^\uparrow$ and its inverse are the only chain transitive trivial automorphisms of $\pwmf$. Thus, under \ocama, there are precisely two chain train transitive automorphisms. However, we claim that neither one is a quotient of the other, so neither one is universal. That \ocama implies $(s^\uparrow)^{-1}$ is not a quotient of $s^\uparrow$ was proved (via Theorem~\ref{thm:farah}) as Theorem 5.7 in \cite{Brian}. The proof given there is easily adapted to show that $s^\uparrow$ is not a quotient of $(s^\uparrow)^{-1}$.

To prove $(2)$, we will need a definition and a few claims.

If $p$ is a mod-finite permutation of $\w$ and has finitely many $\Z$-orbits, then the \emph{index} of $p$, denoted $\dlt(p)$, is the number of $\N$-orbits of $p$, plus the number of backwards $\N$-orbits, plus twice the number of $\Z$-orbits. If $p$ has infinitely many $\Z$-orbits, then we set $\dlt(p) = \infty$.

%\begin{claim}
%If $p$ is a mod-finite permutation of $\w$, then $\dlt(p)$ is equal to the number of $p^*$-invariant clopen subsets $A^* \sub \w^*$ such that $p^* \rest A^*$ is chain transitive.
%\end{claim}
%\begin{proof}[Proof of claim:]
%We mentioned already that the shift map and its inverse are the only chain transitive trivial autohomeomorphisms of $\w^*$. Thus $\w^*$ contains at least $\dlt(p)$ clopen subsets with the stated property, one for each $\N$-orbit and backwards $\N$-orbit, and two for each $\Z$-orbit, because a $\Z$-orbit can be decomposed (by removing one point from the domain of $p$) into an $\N$-orbit and a backwards $\N$-orbit. By Lemma~\ref{lem:ct1}, none of these sets can be further divided into $p^*$-invariant clopen sets, so $\dlt(p)$ is the maximum number of such sets we can get without using the finite cycles of $p$. And it is easy to see that no other such sets arise from the finite cycles of $p$.
%\end{proof}

\begin{claim}
If $p$ and $q$ are mod-finite permutations of $\w$ and $p^* \quotient q^*$, then $\dlt(q) \leq \dlt(p)$.
\end{claim}
\begin{proof}[Proof of claim:]
By removing one point from the domain of $p$ (or $q$), each $\Z$-orbit can be decomposed into an $\N$-orbit and a backwards $\N$-orbit. Thus we may assume that $p$ and $q$ have no $\Z$-orbits, and that $\dlt(p)$ (or $\dlt(q)$) is equal to the number of $\N$-orbits and backwards $\N$-orbits in $p$ (or in $q$, respectively).

Suppose $\dlt(q) > \dlt(p)$ and suppose $Q$ is a quotient mapping from $p^*$ to $q^*$.
Let $A$ be an $\N$-orbit or a backwards $\N$-orbit in $q$. Fix $B \sub \w$ such that $B^* = Q^{-1}(A^*)$. Notice that we cannot have $p(B) \setminus B$ infinite (because $Q$ is a quotient mapping and $q^*(B^*) = A^*$). Thus, by modifying $B$ on a finite set if necessary, we may assume $B$ is a union of orbits in $p$. By Lemma~\ref{lem:farah2}, $B$ is not a union of finite cycles. Thus $B$ contains either an $\N$-orbit or a backwards $\N$-orbit of $p$. This establishes an injection from the set of $\N$-orbits and backwards $\N$-orbits of $q$ to the set of $\N$-orbits and backwards $\N$-orbits of $p$, which proves the claim. 
\end{proof}

\begin{claim}
If $p$ is a mod-finite permutation of $\w$, then $p^*$ is chain recurrent if and only if $\dlt(p) \neq \infty$.
\end{claim}
\begin{proof}[Proof of claim:]
Let $C$ denote the union of all the finite cycles of $p$, and let $A = \w \setminus C$ denote the union of its infinite orbits.
If $\dlt(p) \neq \infty$, then $p^* \rest A^*$ is a finite union of copies of $s^*$ and $(s^*)^{-1}$. Each of these is chain transitive, and $p^* \rest C^*$ is chain recurrent by Lemma~\ref{lem:ct1}. Thus $p^*$ is chain recurrent. 
On the other hand, if $\dlt(p) = \infty$ then $p$ contains infinitely many $\Z$-orbits, and it is easy to check that this implies $p^*$ fails to be chain recurrent.
\end{proof}

From these two claims it follows that there is no universal chain recurrent automorphism. In fact a little more is true: there is no ``jointly universal'' finite family, which means that given any finitely many chain recurrent automorphisms, there is a chain recurrent automorphism (any one with higher index) that is a quotient of none of them.
\end{proof}

Let us now turn to the question of whether $\ocama$ admits a universal automorphism of $\pwmf$. We have seen already that $t^\uparrow$ is no longer universal under \ocama. We will prove:

\begin{theorem}\label{thm:ocamauniv}
Assuming \ocama, 
\begin{enumerate}
\item an automorphism embeds in $t\up$ if and only if it is acyclic.
\item an automorphism embeds in $r\up$ if and only if it is cyclic.
\item every automorphism embeds in either $t^\uparrow$ or $(t \vee r)^\uparrow$.
\end{enumerate}
\end{theorem}

Part $(3)$ of this theorem asserts there is a jointly universal pair of automorphisms under \ocama. We do not know whether this pair can be trimmed down to a single universal automorphism:

\begin{question}\label{q:univ}
$(\ocama)$ Is $(t \vee r)\up$ a universal automorphism?
\end{question}

If $f: X \to X$ and $g: Y \to Y$ are dynamical systems, then a \emph{subquotient mapping} from $g$ to $f$ is a continuous (but not necessarily surjective) function $q: Y \to X$ such that $q \circ g = f \circ q$. 

Let $z$ denote the permutation of $\Z$ mapping $n$ to $n+1$.

\vspace{2mm}\begin{center}
\begin{tikzpicture}[style=thin]

\draw (-3,-1.7) node {};
\draw (-3,-1.3) node {};
\draw (-3,-1.5) node {\large$z$};
\draw (-1.2,-1.5) node {$\dots$};
\draw[fill=black] (0,-1.5) circle (1.5pt);
\draw[fill=black] (1,-1.5) circle (1.5pt);
\draw[fill=black] (2,-1.5) circle (1.5pt);
\draw[fill=black] (3,-1.5) circle (1.5pt);
\draw[fill=black] (4,-1.5) circle (1.5pt);
\draw (5.2,-1.5) node {$\dots$};
\draw[->] (-.5,-1.5) -- (-.15,-1.5);
\draw[->] (.15,-1.5) -- (.85,-1.5);
\draw[->] (1.15,-1.5) -- (1.85,-1.5);
\draw[->] (2.15,-1.5) -- (2.85,-1.5);
\draw[->] (3.15,-1.5) -- (3.85,-1.5);
\draw[->] (4.15,-1.5) -- (4.5,-1.5);

\end{tikzpicture}\end{center}\vspace{2mm}

\begin{lemma}\label{lem:subquotient}
$(\zfc)$ For every autohomeomorphism $\varphi$ of $\w^*$, there is a subquotient mapping from $\b z$ to $\varphi$. Consequently, there is a subquotient mapping from $s^*$ to $\varphi$ and there is a subquotient mapping from $(s^*)^{-1}$ to $\varphi$.
\end{lemma}
\begin{proof}
Let $\varphi$ be any autohomeomorphism of $\w^*$, and let $\U_0$ be any element of $\w^*$. Then
$$\U \mapsto \ulimmm \varphi^n(\U_0)$$
is a subquotient mapping from $\b z$ to $\varphi$. It is clear that $\b z \rest \N^*$ is isomorphic to $s^*$ and that $\b z \rest (\Z \setminus \N)^*$ is isomorphic to $(s^*)^{-1}$, so the ``consequently'' part of the lemma follows.
\end{proof}

\begin{lemma}\label{lem:z}
$(\zfc)$ Both $s^*$ and $(s^*)^{-1}$ are quotients of $z^*$.
\end{lemma}
\begin{proof}
To find a quotient mapping from $z^*$ to $s^*$, let $q_+$ be an isomorphism from $z^* \rest \N^*$ to $s^*$ and let $q_-$ be a subquotient mapping from $z^* \rest (\Z \setminus \N)^*$ to $s^*$. Pasting these together, $q_+ \cup q_-$ is a quotient mapping from $z^*$ to $s^*$. A quotient mapping from $z^*$ to $(s^*)^{-1}$ can be found in exactly the same way.
\end{proof}

\begin{proof}[Proof of Theorem~\ref{thm:ocamauniv}]
To prove $(1)$, let $p$ be an acyclic mod-finite permutation of $\w$. First suppose $p$ has infinitely many $\Z$-orbits and $k$ other orbits (either $\N$-orbits or backwards $\N$-orbits). Let $B$ denote the union of these $k$ orbits. Let $A \sub \w$ denote the union of some (any) $k$ orbits of $t$. By Lemma~\ref{lem:z}, there is a quotient mapping from $t^* \rest A^*$ to $p^* \rest B^*$. Furthermore, $t^* \rest (\w \setminus A)^*$ and $p^* \rest (\w \setminus B)^*$ are isomorphic, since both $t \rest (\w \setminus A)$ and $p \rest (\w \setminus B)$ consist of infinitely many $\Z$-orbits. By pasting together a quotient mapping $t^* \rest A^* \to p^* \rest B^*$ and an isomorphism $t^* \rest (\w \setminus A)^* \to p^* \rest (\w \setminus B)^*$, we obtain a quotient mapping $t^* \to p^*$. 

Next suppose $p$ has only finitely many $\Z$ orbits. Let $k$ denote the total number of orbits of $p$. Let $A$ denote the union of any $k$ of the orbits of $t$, and let $B = \w \times \Z \setminus A$. Using Lemma~\ref{lem:z}, it is clear that there is a quotient mapping from $t^* \rest A^*$ to $p^*$. $B$ consists of infinitely many $\Z$-orbits, and by collapsing them all onto $\Z$ in the natural way we obtain a quotient mapping from $t^* \rest B^*$ to $\b z$. Composing this with a subquotient mapping from $\b z$ to $p^*$, we get a subquotient mapping from $t^* \rest B^*$ to $p^*$. Pasting this together with a quotient mapping from $t^* \rest A^*$ to $p^*$ gives the desired quotient mapping from $t^*$ to $p^*$.

To prove $(2)$, let $c$ be a cyclic permutation of $\w$. It is easy to find a finite-to-one map $q: \bigcup_{n \in \N}\{n\} \times n! \rightarrow \w$ such that $q \circ r = c \circ q$. Then $q^*$ is a quotient mapping from $r^*$ to $c^*$.

For $(3)$, let $p$ be a mod-finite permutation of $\w$. If $p$ is acyclic, then $p^*$ is a quotient of $t^*$ by $(1)$. Suppose $p$ is not acyclic. Let $C$ denote the union of the finite cycles of $p$, and let $A = \w \setminus C$. If $A$ is infinite, then $p^* \rest C^*$ is a quotient of $r^*$ and $p^* \rest A^*$ is a quotient of $t^*$, so pasting quotient mappings together gives a quotient mapping from $(t \vee r)^*$ to $p^*$. If $A$ is finite, then $p^*$ is a quotient of $r^*$ and there is a subquotient mapping from $t^*$ to $p^*$ (there is a natural quotient mapping from $t^*$ to $\b z$, and we may compose this with a subquotient mapping from $\b z$ to $p^*$). Again, pasting these mappings together gives a quotient mapping from $(t \vee r)^*$ to $p^*$.
\end{proof}

Note that this argument shows in \zfc that every trivial automorphism embeds in either $t\up$ or $(t \vee r)\up$. \ocama is only used to eliminate the possibility of nontrivial automorphisms.

\begin{question}
$(\zfc)$ Does every automorphism embed in a trivial one?
\end{question}

If the answer to this question is positive, then, in \zfc, we have a jointly universal pair of automorphisms of $\pwmf$.

\begin{question}\label{q:subquotient}
$(\zfc)$ Is there a subquotient mapping from $r^*$ to $s^*$?
\end{question}

If the answer to Question~\ref{q:subquotient} is positive, then so is the answer to Question~\ref{q:univ}. By the proof of Theorem~\ref{thm:ocamauniv}$(3)$, $(t \vee r)\up$ fails to be universal only if there is no quotient mapping from $(t \vee r)^*$ to some acyclic permutation $p^*$. If there is a subquotient mapping from $r^*$ to $s^*$, then this cannot happen (find a quotient mapping from $t^*$ onto $p^*$, and paste this together with a subquotient mapping from $r^*$ to $s^*$ to $p^*$).

If the answer to both of these questions is positive, then it would imply (in \zfc) that $(t \vee r)\up$ is a universal automorphism of $\pwmf$.

\section{Proof of the main lemma}\label{sec:proof}

This section contains the proof of Theorem~\ref{thm:main}. Before beginning the proof, we will need several definitions and lemmas concerning ultrafilter limits.

\subsection{Limits along an ultrafilter, part I.}

Suppose $X$ is a compact Hausdorff space and $f: D \to X$ is a function. Then there is a unique continuous function $\b f: \b\w \to X$ that extends $f$, the \emph{Stone extension} of $f$. For a $D$-indexed sequence $\seq{x_n}{n \in D}$ of points in $X$ and $\U \in \b D$, we will write $\ulim x_n$ for the image of $\U$ under the Stone extension of the function $n \mapsto x_n$. Equivalently,
$$x = \ulim x_n \ \ \Leftrightarrow \ \ \set{n \in D}{x_n \in V} \in \U \text{ for every neighborhood } V \ni x.$$
Thus every ultrafilter $\U \in D^*$ gives rise to an operator $\ulim$ on $D$-indexed sequences in compact Hausdorff spaces, which picks out a single limit point of the sequence.

Furthermore, these operators commute with continuous functions:

\begin{lemma}\label{lem:ulims1}
Let $X$ be a compact Hausdorff space, and let $\seq{x_n}{n \in D}$ be a $D$-indexed sequence of points in $X$ for some countable set $D$.
\begin{enumerate}
\item If $f: X \to X$ is continuous and $\U \in D^*$, then
$$\textstyle f \!\left( \ulim x_n \right) = \ulim f(x_n).$$
\item If $g: D \to D$ is a finite-to-one function and $\U \in D^*$, then
$$\textstyle \ulim x_{g(n)} = \gulim x_n.$$
\end{enumerate}
\end{lemma}
\noindent A proof can be found in Chapter 3 of \cite{H&S}. The same is true for the following lemma.

\begin{lemma}\label{lem:ulims2}
Let $X$ be a compact Hausdorff space, and let $\seq{x_n}{n \in D}$ be a $D$-indexed sequence of points in $X$ for some countable set $D$.
\begin{enumerate}
\item The map $\U \mapsto \ulim x_n$ is a continuous function $D^* \to X$.
\item $x \in X$ is in the image of this function if and only if every neighborhood of $x$ contains infinitely many of the $x_n$. In particular, the map $\U \mapsto \ulim x_n$ is a surjection if and only if every open subset of $X$ contains infinitely many of the $x_n$.
\end{enumerate}
\end{lemma}

Lemma~\ref{lem:ulims2} tells us how $\U$-limits can be used to define continuous surjections with domain $D^*$. To prove the main lemma, we need to develop a generalization of Lemma~\ref{lem:ulims2} that tells us when $f(\ulim y_n) = \ulim z_n$, even if $f$ is not defined on any of the $y_n$. In the proof of the main lemma, we will have functions $\psi_p$ defined on a subspace $X$ of $[0,1]^{\w_1}$, but with none of the $\psi_p$ defined on $[0,1]^{\w_1} \setminus X$. In that setting, we will need to know how $\psi_p(\ulim y_n)$ behaves, even if some (or all) of the $y_n$ are in $[0,1]^{\w_1} \setminus X$.

\begin{definition}
Let $Y$ be a topological space and let $D$ be a countable set.
 Two $D$-indexed sequences $\seq{x_n}{n \in D}$ and $\seq{y_n}{n \in D}$ of points in $Y$ are \emph{tail-similar} if, for every open cover $\O$ of $Y$, $x_n \approx_\O y_n$ for all but finitely many $n \in D$.
\end{definition}

\begin{lemma}\label{lem:ulims3}
Let $Y$ be a compact Hausdorff space, let $D$ be a countable set, and let $\seq{x_n}{n \in D}$ and $\seq{y_n}{n \in D}$ be $D$-indexed sequences of points in $Y$.
\begin{enumerate}
\item $x = \ulim x_n$ if and only if $\set{n}{x \approx_\O x_n} \in \U$ for every open cover $\O$ of $Y$.
\item $\ulim x_n = \ulim y_n$ if and only if $\set{n}{x_n \approx_\O y_n} \in \U$ for every open cover $\O$ of $Y$.
\item $\ulim x_n = \ulim y_n$ for all $\U \in D^*$ if and only if the sequences $\seq{x_n}{n \in D}$ and $\seq{y_n}{n \in D}$ are tail-similar.
\end{enumerate}
\end{lemma}
\begin{proof}
For $(1)$, the ``if'' part follows easily from the definition of a $\U$-limit. For the ``only if'' part, suppose $\set{n \in D}{x \approx_\O x_n} \notin \U$ for some open cover $\O$. Let $V$ be any member of $\O$ containing $x$. Then $\set{n \in D}{x_n \in V} \notin \U$, which yields $\set{n \in D}{x_n \in Y \setminus V} \in \U$. By the definition of the $\U$-limit, this implies $\ulim x_n \in Y \setminus V$, so that $\ulim x_n \neq x$, completing the proof of $(1)$.
$(2)$ follows easily from $(1)$ and $(3)$ follows easily from $(2)$.
\end{proof}

\begin{definition}
Let $X$ be a closed subset of a compact Hausdorff space $Y$, and $f: X \to X$.
If $\O$ is an open cover of $Y$ and $y,z \in Y$, then we write $y \mapstoof z$
to mean that there is some $x \in X$ with $x \approx_\O y$ and $f(x) \approx_\O z$.
\end{definition}

\noindent  The symbol $y \mapstoof z$ is read ``$f$ maps $y$ to $z$ modulo $\O$" and can be thought of as asserting that the expression ``$f(y)=z$'', though it may formally be mere nonsense, is approximately correct.

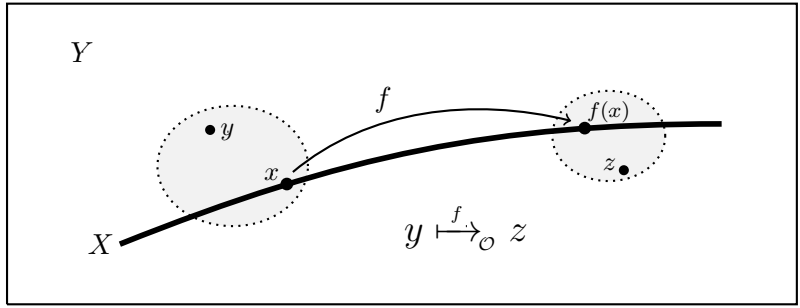
\vspace{2mm}\begin{center}
\begin{tikzpicture}[style=thick, xscale=1,yscale=.8]

\draw (-1.5,-2) -- (-1.5,3) -- (9,3) -- (9,-2) -- (-1.5,-2);
\draw[dotted,fill=black!5] (1.5,.3) circle (1); 
\draw[dotted,fill=black!5] (6.5,.8) circle (.75);
\draw[line width=.75mm] (0,-1) .. controls (3,.5) and (5,1) .. (8,1);
\draw (1.2,.9) node [circle, draw, fill=black, inner sep=0pt, minimum width=3pt] {} -- (1.2,.9);
\draw (1.43,.87) node {\small $y$} -- (1.43,.87);
\draw (6.7,.23) node [circle, draw, fill=black, inner sep=0pt, minimum width=3pt] {} -- (6.7,.23);
\draw (6.51,.32) node {\small $z$} -- (6.51,.32);
\draw (3.5,1.4) node {$f$} -- (3.5,1.4);
\draw[->] (2.3,.2) .. controls (3.3,1.25) and (4.7,1.45) .. (6,1.05);
\draw (2.22,0) node [circle, draw, fill=black, inner sep=0pt, minimum width=4pt] {} -- (2.22,0);
\draw (2.02,.16) node {\small $x$} -- (2.02,.16);
\draw (6.18,.93) node [circle, draw, fill=black, inner sep=0pt, minimum width=4pt] {} -- (6.18,.93);
\draw (6.5,1.2) node {\scriptsize $f(x)$} -- (6.5,1.2);
\draw (-.25,-1) node {$X$} -- (-.25,-1);
\draw (-.5,2.2) node {$Y$} -- (-.5,2.2);
\draw (4.6,-.7) node {\Large $y \mapstoof z$} -- (4.6,-.7);

\end{tikzpicture}\end{center}\vspace{2mm}

\begin{lemma}\label{lem:ulims4}
Let $X$ be a closed subspace of a compact Hausdorff space $Y$, and let $D$ be a countable set. Let $f: X \to X$ be continuous and let $\seq{y_n}{n \in D}$ and $\seq{z_n}{n \in D}$ be $D$-indexed sequences of points in $Y$. The following are equivalent:
\begin{enumerate}
\item For every open cover $\O$ of $Y$, $y_n \mapstoof z_n$ for all but finitely many $n \in D$.
\item For every $\U \in D^*$, $\ulim y_n$ and $\ulim z_n$ are both in $X$, and $$\textstyle f(\ulim y_n) = \ulim z_n.$$
\end{enumerate} 
\end{lemma}
\begin{proof}
To prove $(1)$ implies $(2)$, first note that if $(1)$ holds, then the definition of a $\U$-limit implies (using the fact that $\U$ is non-principal) that both $\ulim y_n$ and $\ulim z_n$ are in $X$. Let $y = \ulim y_n$ and $z = f(y)$, and let $V$ be an open neighborhood of $z$ (in $Y$). We need to show that $\set{n}{z_n \in V} \in \U$; then, because $V$ is arbitrary, we will know $\ulim z_n = z$ by definition.

To show $\set{n}{z_n \in V} \in \U$, we make use of Lemma 3.1 from \cite{Brian}, a special case of which states:
\begin{itemize}
\item There is an open cover $\O$ of $Y$ such that if $y',z' \in Y$ with $y' \approx_\O y$ and $y' \mapstoof z'$, then $z' \in V$.
\end{itemize}
(The lemma in \cite{Brian} is stated only for $Y = [0,1]^\dlt$, but one can easily check that the proof does not depend on this.) Heuristically, this lemma just expresses the continuity of $f$, which requires that points near $y$ map to points near $z$. 

Let $\O$ be an open cover of $Y$ as described above. By the definition of a $\U$-limit, we have
$\set{n \in D}{y_n \approx_\O y} \in \U.$
Because $y_n \mapstoof z_n$ for all but finitely many $n \in D$, and because $\U$ is a non-principal ultrafilter on $D$,
$$\set{n \in D}{y_n \approx_\O y \text{ and } y_n \mapstoof z_n} \in \U.$$
By our choice of $\O$,
$$\set{n \in D}{z_n \in V} \supseteq \set{n \in D}{y_n \approx_\O y \text{ and } y_n \mapstoof z_n}$$
so that $\set{n}{z_n \in V} \in \U$ as required.

For the converse, suppose that $\ulim y_n$ and $\ulim z_n$ are both in $X$ for every $\U \in D^*$, but $(1)$ does not hold. We will show that this implies $\textstyle f(\ulim y_n) \neq \ulim z_n$ for some $\U \in D^*$.

If $(1)$ fails, then there is some open cover $\O$ of $Y$ and some infinite $A \sub D$ such that $\neg (y_n \mapstoof z_n)$ for all $n \in A$. Fix $\U \in A^*$ and observe that, by the definition of a $\U$-limit,
$$B = \set{n \in D}{\textstyle y_n \approx_\O \ulim y_n} \in \U.$$
If $n \in B$ and $z_n \approx_\O f(\ulim y_n)$, then $y_n \mapstoof z_n$. Thus, for $n \in B \cap A$, we have $z_n \not\approx_\O f(\ulim y_n)$. But $B \cap A \in \U$, so this implies $\ulim z_n \neq f(\ulim y_n)$ as desired.
\end{proof}

\subsection{$\varphi$-like sequences revisited.}

In this subsection we extend the notion of $\varphi$-like sequences to a broader context, where $X$ is no longer required to be metrizable, and where the sequence of points may lie in some space $Y \supseteq X$. 

\begin{definition}\label{def:philike}
Let $X$ be a closed subset of a compact Hausdorff space $Y$, and let $D$ be a countable set.
\begin{itemize}
\item Let $f: X \to X$ be continuous, and let $p: D \to D$ be any function.
\begin{itemize}
\item[$\circ$] A $D$-indexed sequence $\seq{y_n}{n \in D}$ of points in $Y$ is said to be \emph{$p$-like for $\O$ with respect to $f$}, where $\O$ is an open cover of $Y$, if
$$y_n \mapstoo y_{\varphi(n)} \qquad \text{for all but finitely many }n \in D.$$
When $f$ is clear from context, we say simply that the sequence is $p$-like for $\O$.
\item[$\circ$] If $\seq{y_n}{n \in D}$ is $p$-like for $\O$ for every open cover $\O$ of $Y$, then we say that $\seq{y_n}{n \in D}$ is \emph{$p$-lilke (with respect to $f$)}.
\end{itemize}
\item Let $\psi: S \times X \to X$ be an $S$-flow, and let $\varphi: S \times D \to D$ be an action of $S$ on $D$.
\begin{itemize}
\item[$\circ$] A $D$-indexed sequence $\seq{y_n}{n \in D}$ of points in $Y$ is said to be \emph{$\varphi$-like for $\O$ with respect to $\psi$}, where $\O$ is an open cover of $Y$, if for every $p \in S$ it is $\varphi_p$-like for $\O$ with respect to $\psi_p$.
When $\psi$ is clear from context, we say simply that the sequence is $\varphi$-like for $\O$.
\item[$\circ$] If $\seq{y_n}{n \in D}$ is $\varphi$-like for $\O$ for every open cover $\O$ of $Y$, then we say that $\seq{y_n}{n \in D}$ is {$\varphi$-lilke (with respect to $f$)}.
\end{itemize}
\end{itemize}
\end{definition}

Observe that this new definition contains the old one as a special case, namely when $X=Y$ and $X$ is metrizable.

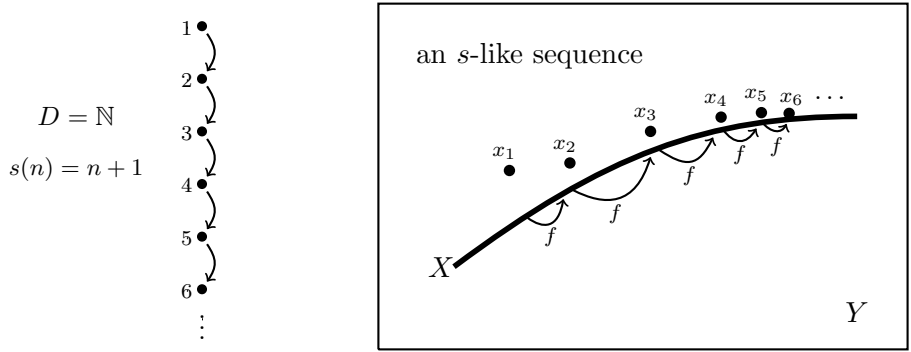
\vspace{1mm}
\begin{center}
\begin{tikzpicture}[style=thick, xscale=.67,yscale=1]

\draw (-5,-.6) node [circle, draw, fill=black, inner sep=0pt, minimum width=3pt] {} -- (-5,-.6);
\draw (-5,.1) node [circle, draw, fill=black, inner sep=0pt, minimum width=3pt] {} -- (-5,.1);
\draw (-5,.8) node [circle, draw, fill=black, inner sep=0pt, minimum width=3pt] {} -- (-5,.8);
\draw (-5,1.5) node [circle, draw, fill=black, inner sep=0pt, minimum width=3pt] {} -- (-5,1.5);
\draw (-5,2.2) node [circle, draw, fill=black, inner sep=0pt, minimum width=3pt] {} -- (-5,2.2);
\draw (-5,-1.3) node [circle, draw, fill=black, inner sep=0pt, minimum width=3pt] {} -- (-5,-1.3);
\draw[->] (-4.9,2.1) .. controls (-4.7,1.9) and (-4.7,1.8) .. (-4.9,1.6);
\draw[->] (-4.9,1.4) .. controls (-4.7,1.2) and (-4.7,1.1) .. (-4.9,.9);
\draw[->] (-4.9,.7) .. controls (-4.7,.5) and (-4.7,.4) .. (-4.9,.2);
\draw[->] (-4.9,0) .. controls (-4.7,-.2) and (-4.7,-.3) .. (-4.9,-.5);
\draw[->] (-4.9,-.7) .. controls (-4.7,-.9) and (-4.7,-1) .. (-4.9,-1.2);
\draw (-5,-1.7) node {\small $\vdots$} -- (-5,-1.7);
\draw (-5.3,2.18) node {\scriptsize $1$} -- (-5.3,2.18);
\draw (-5.3,1.48) node {\scriptsize $2$} -- (-5.3,1.48);
\draw (-5.3,.78) node {\scriptsize $3$} -- (-5.3,.78);
\draw (-5.3,.08) node {\scriptsize $4$} -- (-5.3,.08);
\draw (-5.3,-.62) node {\scriptsize $5$} -- (-5.3,-.62);
\draw (-5.3,-1.32) node {\scriptsize $6$} -- (-5.3,-1.32);
\node at (-7.5,1) {\small $D = \N$};
\node at (-7.5,.3) {\footnotesize $s(n) = n+1$};

\draw (-1.5,-2.1) -- (-1.5,2.5) -- (9,2.5) -- (9,-2.1) -- (-1.5,-2.1);
\node at (8,-1.6) {$Y$};
\node at (1.5,1.8) {an $s$-like sequence};
\draw[line width=.75mm] (0,-1) .. controls (3,.5) and (5,1) .. (8,1);
\draw (-.25,-1) node {$X$} -- (-.25,-1);

%\draw (1.32,-.38) node [circle, draw, fill=black, inner sep=0pt, minimum width=4pt] {} -- (1.32,-.38);
\draw[->] (1.46,-.35) .. controls (1.77,-.5) and (2.07,-.5) .. (2.17,-.1);
\draw (1.92,-.6) node {\scriptsize $f$} -- (1.92,-.6);
%\draw (2.22,0) node [circle, draw, fill=black, inner sep=0pt, minimum width=4pt] {} -- (2.22,0);
\draw[->] (2.36,.02) .. controls (3,-.2) and (3.6,-.2) .. (3.9,.46);
\draw (3.2,-.29) node {\scriptsize $f$} -- (3.2,-.29);
%\draw (3.95,.55) node [circle, draw, fill=black, inner sep=0pt, minimum width=4pt] {} -- (3.95,.55);
\draw[->] (4.08,.55) .. controls (4.6,.32) and (4.85,.32) .. (5.14,.72);
\draw (4.67,.21) node {\scriptsize $f$} -- (4.67,.21);
%\draw (5.2,.81) node [circle, draw, fill=black, inner sep=0pt, minimum width=4pt] {} -- (5.2,.81);
\draw[->] (5.33,.83) .. controls (5.56,.62) and (5.8,.62) .. (6,.84);
\draw (5.67,.49) node {\scriptsize $f$} -- (5.67,.49);
%\draw (6.08,.92) node [circle, draw, fill=black, inner sep=0pt, minimum width=4pt] {} -- (6.08,.92);
\draw[->] (6.1,.93) .. controls (6.3,.77) and (6.5,.77) .. (6.6,.9);
\node at (6.4,.59) {\scriptsize $f$};
%\draw (6.08,.92) node [circle, draw, fill=black, inner sep=0pt, minimum width=4pt] {} -- (6.08,.92);

\node[circle, draw, fill=black, inner sep=0pt, minimum width=3.5pt] at (1.1,.28) {};
\node at (1,.51) {\scriptsize $x_1$};
\node[circle, draw, fill=black, inner sep=0pt, minimum width=3.5pt] at (2.3,.38) {};
\node at (2.2,.61) {\scriptsize $x_2$};
\node[circle, draw, fill=black, inner sep=0pt, minimum width=3.5pt] at (3.9,.8) {};
\node at (3.8,1.03) {\scriptsize $x_3$};
\node[circle, draw, fill=black, inner sep=0pt, minimum width=3.5pt] at (5.3,.99) {};
\node at (5.2,1.22) {\scriptsize $x_4$};
\node[circle, draw, fill=black, inner sep=0pt, minimum width=3.5pt] at (6.1,1.05) {};
\node at (6,1.28) {\scriptsize $x_5$};
\node[circle, draw, fill=black, inner sep=0pt, minimum width=3.5pt] at (6.65,1.04) {};
\node at (6.7,1.23) {\scriptsize $x_6$};
\node at (7.5,1.25) {\small $\dots$};

\end{tikzpicture}\end{center}\vspace{1mm}

%The following lemma is the main topological lemma in the proof of our theorem, and the primary result of this section. It tells us precisely when a $D$-indexed sequence of points in $Y$ induces a quotient mapping from a trivial flow $\varphi^*$ on $D^*$ to some flow $\psi$ defined on $X \sub Y$.

\begin{lemma}\label{lem:psilike}
Let $X$ be a closed subspace of some compact Hausdorff space $Y$, let $\psi$ be an $S$-flow on $X$, and let $\seq{x_n}{n \in D}$ be a $D$-indexed sequence of points in $Y$ for some countable set $D$. If $\varphi$ is an action of $S$ on $D$, then
\begin{enumerate}
\item[$(\dagger)$] The map $\pi: D^* \to Y$ defined by $\pi(\U) = \ulim x_n$ is a quotient mapping from $\varphi^*$ to $\psi$ if and only if
\begin{enumerate}
\item $\pi[D^*] = X$, and
\item $\seq{x_n}{n \in D}$ is $\varphi$-like.
\end{enumerate}
\end{enumerate}
\end{lemma}
\begin{proof}
We will prove the ``if'' direction of $(\dagger)$ first. Using $(a)$ and Lemma~\ref{lem:ulims2}, $\pi$ is a continuous surjection from $D^*$ onto $X$. We need to show that $\pi$ preserves the action of $S$, in the sense that $\pi \circ \varphi^*_p = \psi_p \circ \pi$ for all $p \in S$. This is a direct application of Lemma~\ref{lem:ulims4}:
\begin{align*}
\psi_p \circ \pi(\U) & = \psi_p(\pi(\U)) = \psi_p(\textstyle \ulim x_n) \\
& = \ulim x_{\varphi_p(n)} = \phiulim x_n \\
& = \pi(\varphi^*_p(\U)) = \pi \circ \varphi^*_p(\U)
\end{align*}
(Lemma~\ref{lem:ulims4} was applied to get the third equality, and Lemma~\ref{lem:ulims1}(2) to get the fourth).

For the ``only if'' direction, suppose $\pi$ is a quotient mapping and let us prove that (b) holds (note that (a) holds by the definition of a quotient mapping). Aiming for a contradiction, suppose $\seq{x_n}{n \in D}$ is not $\varphi$-like. Fix $p \in S$, an open cover $\O$ of $Y$, and an infinite $A \sub D$ such that, for every $n \in A$, it is false that $x_n \mapstoo x_{\varphi(n)}$; that is, if $n \in A$ and $x_n \approx_{\O} x \in X$, then $\psi_p(x) \not\approx_{\O} x_{\varphi(n)}$.

Let $\U \in D^*$ with $A \in \U$, and let $x = \ulim x_n$. Fix $O \in \O$ with $x \in O$, set $B = \set{n \in A}{x_n \in O}$, and observe that $B \in \U$. Now $x_n \approx_\O x$ for all $n \in B$, and by our choice of $\O$ this means $\psi_p(x) \not\approx_\O x_{\varphi(n)}$ for all $n \in B$. In particular, if $\psi_g(x) \in V \in \O$ then $\set{n \in B}{x_{\varphi_g(n)} \in V} \sub \N \setminus B$. Thus $\set{n \in B}{x_{\varphi_p(n)} \in V} \notin \U$ and $\ulim x_{\varphi_p(n)} \neq \psi_p(x)$. Hence
$$\pi(\varphi^*p(\U)) = \ulim x_{\varphi_p(n)} \neq \psi_p(x) = \psi_p(\pi(\U))$$
so that $\pi$ is not a quotient mapping.
\end{proof}

\subsection{The proof} 

We are finally ready to prove the Main Lemma. We reproduce the statement of the lemma here for convenience:

\begin{theorem34}
Let $S$ be a countable discrete semigroup, let $X$ be a compact Hausdorff space of weight $\leq\!\aleph_1$, and let $D$ be a countable set. Let $\varphi$ be a separately finite-to-one action of $S$ on $D$, and let $\psi: S \times X \to X$ be an $S$-flow. The following are equivalent:
\begin{enumerate}
\item $\psi$ is a quotient of $\varphi^*$.
\item Every metrizable quotient of $\psi$ is a quotient of $\varphi^*$.
\item Some metrizable reflection of $\psi$ is a quotient of $\varphi^*$.
\item Every metrizable quotient of $\psi$ contains a tail-dense $\varphi$-like sequence.
\item Some metrizable reflection of $\psi$ contains a tail-dense $\varphi$-like sequence.
\end{enumerate}
\end{theorem34}

\begin{proof}
The structure of the proof is as follows. The implications $(3) \Rightarrow (2)$ and $(5) \Rightarrow (4)$ are trivial. We will first prove that $(1) \Rightarrow (3)$, and then prove a lemma showing that $(2) \Leftrightarrow (4)$ and $(3) \Leftrightarrow (5)$. Finally, we will show $(4) \Rightarrow (1)$ to finish the proof. Showing $(4) \Rightarrow (1)$ is the longest and most difficult part of the proof.

To show that $(1)$ implies $(3)$, it is enough to show that every metrizable reflection of $\psi$ is a quotient of $\psi$, because the composition of two quotient mappings is again a quotient mapping. But this is obvious: if $M$ is a countable elementary submodel of $H$, so that $\psi^M$ is a metrizable reflection of $\psi$, then the natural projection $\Pi_{\dlt^M}: X \to X^M$ is a quotient mapping.

The following lemma immediately implies that $(2) \Leftrightarrow (4)$ and $(3) \Leftrightarrow (5)$.

\begin{lemma}
Suppose $Z$ is a compact metric space and $\mu: G \times Z \to Z$ is a $G$-flow. Let $\varphi: G \times D \to D$ be an action of $G$ on a countable set $D$. Then $\mu$ is a quotient of $\varphi^*$ if and only if $Z$ contains a tail-dense $\varphi$-like sequence. 
\end{lemma}
\begin{proof}
In order to prove our lemma, we will need the following folklore result:

\vspace{2mm}

\noindent \emph{Fact: } If $F: D^* \to Z$ is a continuous function, then there is a $D$-indexed sequence $\seq{z_n}{n \in D}$ of points in $Z$ such that $F(\U) = \ulim z_n$ for all $\U \in D^*$.

\vspace{2mm}

\begin{proof}[Proof of fact:] This fact is well-known in the case $Z = [0,1]$ (indeed, some authors take this as the defining property of $D^*$). One way to prove this fact in general is to apply Urysohn's theorem, which allows us to view $Z$ (up to homeomorphism) as a closed subspace of $[0,1]^\w$. For each $m \in \w$, the map $\pi_m \circ F$ is a continuous function $D^* \to [0,1]$, so the special case of the fact allows us to find a $D$-indexed sequence $\seq{s_n^m}{n \in D}$ of points in $Z$ such that $\pi_m \circ F(\U) = \ulim s_n^m$ for all $\U \in D^*$. Taking $s_n = \Delta_{n \in \w}s^m_n$ now gives us a $D$-indexed sequence $\seq{s_n}{n \in D}$ of points in $[0,1]^\w$ (but not necessarily in $Z$) such that $F(\U) = \ulim s_n$ for all $\U \in D^*$. Finally, we may get $z_n$ from $s_n$ by using the metrizability of $[0,1]^\w$. Define $z_n$ to be any point of $Z$ such that $\mathrm{dist}(z_n,s_n) = \mathrm{dist}(Z,s_n)$ (such a point exists because $Z$ is closed). One easily shows that the sequences $\seq{z_n}{n \in D}$ and $\seq{s_n}{n \in D}$ are tail-similar, so that Lemma~\ref{lem:ulims3}(3) completes the proof.
\end{proof}

Returning to the proof of the lemma, suppose $\pi: D^* \to Z$ is a quotient mapping from $\varphi^*$ to $\mu$. Using the above fact, there is a $D$-indexed sequence $\seq{z_n}{n \in D}$ of points in $Z$ such that $\pi(\U) = \ulim z_n$ for all $\U \in D^*$. By Lemma~\ref{lem:psilike}, the sequence $\seq{z_n}{n \in D}$ is $\varphi$-like, which proves the forward direction. Conversely, if a sequence $\seq{z_n}{n \in D}$ of points in $Z$ is $\varphi$-like, then Lemma~\ref{lem:psilike} states that the function $\U \mapsto \ulim z_n$ is a quotient mapping from $\varphi^*$ to $\mu$. 
\end{proof}

\vspace{3mm}

It remains to show that $(4)$ implies $(1)$. Let $M$ be a countable elementary submodel of $H$ with $X, \psi \in M$, so that $\psi^M$ is a metrizable reflection of $\psi$, and suppose that $X^M$ contains a $\varphi$-like sequence of points. By Lemma~\ref{lem:psilike}, it suffices to construct a $\varphi$-like sequence of points in $[0,1]^{\w_1}$; this is what we shall do.

We will construct a map $q_\xi: D \to [0,1]$ for every $\xi < \w_1$ via a length-$\w_1$ transfinite recursion. In the end, the diagonal mapping $Q = \Delta_{\xi < \w_1}q_\xi$ will define a $\varphi$-like sequence $\seq{Q(n)}{n \in \w}$ in $[0,1]^{\w_1}$.

Before beginning this construction, we will need a little more terminology. Let us say that $\O$ is a \emph{nice} open cover for $[0,1]^{\w_1}$ if it is an open cover consisting of finitely many sets of the form
$$\pi_{\a_1}^{-1}[I_1] \cap \pi_{\a_2}^{-1}[I_2] \cap \dots \cap \pi_{\a_k}^{-1}[I_k]$$
where each $\a_j$ is an ordinal $<\! \w_1$ and each $I_j$ is an open interval in $[0,1]$ with rational endpoints.

Observe that the nice open covers suffice to describe the topology of $[0,1]^{\w_1}$ (in the sense that every open cover of $[0,1]^{\w_1}$ is refined by a nice open cover). This proves the following simple but useful observation:

\begin{observation1} \label{ob:nice}
If a sequence $\seq{x_n}{n \in D}$ of points in $[0,1]^{\w_1}$ is $\varphi$-like for every nice open cover, then it is $\varphi$-like.
\end{observation1}

Also observe that every nice open cover $\O$ is (uniformly) definable from a finite list of ordinals. For example, the open set above is definable from $\a_1,\a_2,\dots,\a_k$ (we do not need to mention the intervals $I_1,I_2,\dots,I_k$ in a definition of this set, because the intervals have rational endpoints, and rational numbers are definable without parameters), and $\O$ consists of finitely many open sets like this one. We say that a nice open cover is defined using ordinals $<\!\dlt$ if each ordinal $\a_i$ appearing as above in the definition of some $U \in \O$ is $<\!\dlt$.

Finally, observe that whether a sequence is $\varphi$-like for some nice open cover $\O$ depends only on the projection of that sequence onto the ordinals used to define $\O$. Specifically, we have:

\begin{observation2}\label{ob:bdd}
Let $\dlt \leq \w_1$ and let $\O$ be a nice open cover of $[0,1]^{\w_1}$ defined using only ordinals $<\!\dlt$. If $\seq{x_n}{n \in D}$ and $\seq{y_n}{n \in D}$ are $D$-indexed sequences in $[0,1]^{\w_1}$ such that
$$\Pi_{\dlt}(x_n) = \Pi_{\dlt}(y_n) \qquad \text{for all } n \in D,$$
then either both sequences are $\varphi$-like for $\O$ or neither is.
\end{observation2}

Abusing our terminology slightly, if $\dlt < \w_1$ then a sequence $\seq{x_n}{n \in D}$ of points in $[0,1]^\dlt$ will be called $\varphi$-like for some nice open cover $\O$ of $[0,1]^{\w_1}$ provided that $\O$ is defined using ordinals $<\!\dlt$, and that any sequence $\seq{y_n}{n \in D}$ of points in $[0,1]^{\w_1}$ with 
$$\Pi_{\dlt'}(x_n) = \Pi_{\dlt'}(y_n) \qquad \text{for all } n \in D$$
is $\varphi$-like for $\O$. Similarly, a sequence of points in $[0,1]^\dlt$ is called $\varphi$-like if it is $\varphi$-like for every nice open cover of $[0,1]^{\w_1}$ defined using ordinals $<\!\dlt$.

For the remainder of the proof, it will be convenient to identify $D$ with $\w \setminus \{0\} = \N$. This sacrifices no generality as, until now, $D$ was an arbitrary countable set.

With these things in place, let us turn to our recursive construction. Using the L\"{o}wenheim-Skolem theorem, fix a sequence $\seq{M_\a}{\a < \w_1}$ of countable elementary submodels of $H$ such that
\begin{enumerate}
\item $M_0 = M$.
\item for each $\a$, $\seq{M_\b}{\b \leq \a} \in M_{\a+1}$.
\item for limit $\a$, $M_\a = \bigcup_{\b < \a}M_\b$.
\end{enumerate}
For convenience, we will write $\dlt_\a$ instead of $\dlt^{M_\a}$ for each $\a < \w_1$.

As stated above, we will construct a map $q_\xi: D \to [0,1]$ for every $\xi < \w_1$ via a length-$\w_1$ transfinite recursion. Step $\a$ of the recursion will be used to construct simultaneously all the maps $q_\xi$ with $\xi \in \dlt_\a \setminus \bigcup_{\b < \a}\dlt_\b$. The construction will ensure that at the end of stage $\a$, the map $\Delta_{\xi < \dlt_\a}q_\xi$ defines a $\varphi$-like sequence of points in $[0,1]^{\dlt_\a}$.

For the base step of the recursion, we use our assumption that the metrizable reflection $\psi^M$ of $\psi$ contains a $\varphi$-like sequence. Let $\seq{r_n}{n \in D}$ be a $\varphi$-like sequence of points in $X^{M_0}$, and recall that $X^{M_0} \sub [0,1]^{\dlt_0}$. Taking $q_\xi(n) = \pi_\xi(r_n)$, the map $\Delta_{\xi < \dlt_0}q_\xi$ defines a $\varphi$-like sequence of points in $[0,1]^{\dlt_0}$ as desired.

Without loss of generlity, we may (and do) assume that each of the real numbers $q_\xi(n)$, for $\xi < \dlt_0$ and $n \in D$, is rational. To justify this assumption, use the fact that $[0,1]^{\dlt_0}$ is metrizable to choose for each $r_n$ a point $r_n' \in M_0$ with rational coordinates such that $\mathrm{dist}(r_n,r_n') < \nicefrac{1}{n}$. These two sequences are clearly similar. Hence $\seq{r'_n}{n \in D}$ is $\varphi$-like, and replacing $\seq{r_n}{n \in D}$ with $\seq{r_n'}{n \in D}$ in the definition of the $q_\xi$ makes each number of the form $q_\xi(n)$ rational.

[Note: The reason for making the $q_\xi(n)$ all rational is that we will need $q_\xi(n) \in M_0$ in order to make use of the elementarity assumption $M_0 \preceq M_1$ at the next stage of the recursion. Making each $q_\xi(n)$ rational is just a convenient way to accomplish this.]

This completes the base step of the recursion. Observe that the construction of the $q_\xi$, $\xi < \dlt_0$, can be carried out in $M_1$, because $M_0 \in M_1$.

%Taking $q_\b(n) = \pi_\b(s_n)$ for all $\b < \dlt_0$ and $n \in D$ completes the base step of the recursion. Observe that the construction of the $q_\b$, $\b < \dlt_0$, can be carried out in $M_1$, because $M_0 \in M_1$.

%Note that part of the above argument takes place ``outside'' of the model $M_0$, because we use the fact that $X^{M_0}$ is metrizable. Thus, even though each point $s_n$ is in $M_0$, there is nothing to guarantee that the sequence $\seq{s_n}{n \in D}$ is in $M_0$. This should not be surprising: if we already had $\seq{s_n}{n \in D} \in M_0$, then we could easily prove (within $M_0$) that $\psi$ is a quotient of $\varphi^*$, and then apply elementarity to claim that the same is true in $H$. In other words, we would be done and there would be no need to proceed with the rest of the construction!

For later stages of the recursion, we assume that the following three inductive hypotheses hold at the end of stage $\a$: letting $y^\a_n = \Delta_{\xi < \dlt_\a}q_\xi(n)$ for convenience, then
\begin{itemize}
\item [(H1)] the sequence $\seq{y^\a_n}{n \in \w}$ is $\varphi$-like in $[0,1]^{\dlt_\a}$.
\item [(H2)] the sequence $\seq{y^\a_n}{n \in \w}$ is a member of $M_{\a+1}$.
\item [(H3)] for each $\xi < \a$, $q_\xi(n)$ is a rational number.
\end{itemize}
Notice that these hypotheses hold at the end of stage $\a=0$ described above.

For limit $\a$ there is nothing to do: clause $(3)$ in our choice of the $M_\a$ guarantees that $\dlt_\a = \bigcup_{\b < \a}\dlt_\b$ for all limit ordinals $\a$, so at stage $\a$ the maps $q_\xi$ are already defined for every $\xi < \dlt_\a$. The hypotheses $(\text{H}2)$ and $(\text{H}3)$ clearly hold at $\a$ if they hold for every $\b < \a$. For $(\text{H}1)$, note that every nice open cover of $[0,1]^{\w_1}$ is defined from only finitely many ordinals. Thus, because $\dlt_\a = \bigcup_{\b < \a}\dlt_\b$, any nice open cover of $[0,1]^{\w_1}$ defined from ordinals $<\!\dlt_\a$ is already defined from ordinals $<\!\dlt_\b$ for some $\b < \a$. $(\text{H}1)$ at $\a$ now follows from Observation 2 and the fact that $(\text{H}1)$ holds at $\b$ for every $\b < \a$.

For the successor stages of the recursion, fix $\a < \w_1$ and suppose the functions $q_\xi$ have already been constructed for every $\xi < \dlt_\a$. For each $n$, let $y^\a_n = \Delta_{\xi < \dlt_\a}q_\xi(n)$.
We will show how to obtain $q_\xi$ for $\dlt_\a \leq \xi < \dlt_{\a+1}$.

There are only countable many nice open covers of $[0,1]^{\w_1}$ defined using ordinals $<\!\dlt_{\a+1}$. Also, any finitely many of these covers have a common refinement that is also a nice open cover of $[0,1]^{\w_1}$ defined using ordinals $<\!\dlt_{\a+1}$. Thus we may find a countable sequence $\seq{\O_m}{m < \w}$ of nice open covers of $[0,1]^{\w_1}$ defined using ordinals $<\!\dlt_{\a+1}$ such that
\begin{enumerate}
\item $\O_m \in M_{\a+1}$ for every $m \in \w$,
\item $\O_m$ refines $\O_\ell$ whenever $\ell \leq m$, and
\item if $\O$ is any basic open cover of $[0,1]^{\dlt_{\a+1}}$, then some $\O_m$ refines $\O$.
\end{enumerate}
Note that this part of the construction occurs ``outside'' $M_{\a+1}$, because we are using the fact that $\dlt_{\a+1}$ is countable.

Fix $m \in \w$ and consider $\O_m$. The set of ordinals used in the definition of $\O_m$ is finite and may be split into two parts: those ordinals that are below $\dlt_\a$, and those that are in the interval $[\dlt_\a,\dlt_{\a+1})$. Let us call these two finite sets of ordinals $F_m^\downarrow$ and $F_m^\uparrow$, respectively.

The ordinals in $F_m^\uparrow$ are not in $M_\a$, but for each ordinal $\xi \in F_m^\uparrow$ we may use elementarity to find an ``avatar'' ordinal $\z \in M_{\a}$ that acts like $\xi$, at least with respect to some prescribed first-order formula. The idea behind defining the $q_\xi$ for $\xi \geq \dlt_\a$ is to find a sequence of increasingly faithful avatars, and then to define $q_\xi$ by diagonalizing across the avatar functions $q_\z$.

More precisely, let $F_m^\uparrow = \set{\xi_i}{i \leq \ell_m}$. For any first-order formula
$$\Phi(\xi_1,\xi_2,\dots,\xi_{\ell_m},a)$$
where $a$ is a parameter from $M_\a$, we may apply elementarity to find a set $E = \set{\z_i}{i \leq \ell_m} \sub M_\a$ such that
$$M_{\a+1} \models \Phi(\xi_1,\xi_2,\dots,\xi_{\ell_m},a) \qquad \Leftrightarrow \qquad M_\a \models \Phi(\z_1,\z_2,\dots,\z_{\ell_m},a)$$
Formulas containing more and more information about $X$ and $\psi$ will yield increasingly faithful avatars of the ordinals in $F_m^\uparrow$.

Let us enumerate $S = \{p_0,p_1,p_2,p_3,\dots\}$. For every $m < \w$ it is possible to write down in the language of first-order logic a (very long) formula $\Phi^m(x_1,x_2,\dots,x_{\ell_m},X,\psi)$ that does all of the following:
\begin{enumerate}
\item $\Phi^m$ declares that each $x_i$ is a countable ordinal.
\item $\Phi^m$ defines a nice open cover in terms of $F^0_m$ and $x_1,x_2,\dots,x_{\ell_m}$. This is done in the natural way, so that the open cover defined by setting $x_1=\xi_1,x_2=\xi_2,\dots,x_{\ell_m}=\xi_{\ell_m}$ is $\O_m$.
\item $\Phi^m$ declares that certain properties hold of the nice open cover it defines and how that open cover interacts with $X$ and with the maps $\psi_{p_0},\psi_{p_1},\dots,\psi_{p_m}$. This is again done in the natural way, so that $\Phi^m(\xi_1,\xi_2,\dots,\xi_{\ell_m},X,\psi)$ declares the complete list of the following combinatorial properties of $\O_m$:
\begin{enumerate}
\item for all $\mathcal J \sub \O_m$, $\Phi^m$ asserts either that $\bigcap \mathcal J \cap X = \0$ or that $\bigcap \mathcal J \cap X \neq \0$,
\item if $\mathcal J \sub \O_m$, $\bigcap \mathcal J \cap X \neq \0$, $U \in \O_m$, and $i \leq m$, then $\Phi^m$ asserts either that $\psi_{p_i}[\bigcup \mathcal J \cap X] \cap U = \0$ or that $\psi_{p_i}[\bigcup \mathcal J \cap X] \cap U \neq \0$.
\end{enumerate}
\end{enumerate}
Put simply, the formula $\Phi^m$ records the definition of the open cover $\O_m$ and its behavior with respect to the maps $\psi_{p_0},\psi_{p_1},\dots,\psi_{p_m}$.

If $E = \seq{\z_i}{i \leq \ell_m}$ is a finite sequence of ordinals $<\!\dlt_{\a+1}$ such that $\Phi^m(\z_1,\z_2,\dots,\z_{\ell_m},X,\psi)$ holds, let us write $\O^m(E)$ for the basic open cover defined by $\Phi^m$. For example, $\O^m(F_m^\uparrow) = \O_m$ for all $k$.

Given two points $x,x'$ in $[0,1]^{\w_1}$, the information contained in $(2)$ is enough to determine precisely which elements of $\O^m(E)$ contain each of $x$ and $x'$. Once that is known, the information in $(3)$ is enough to determine whether or not $x \mapstooof x'$ for any $p = p_0,p_1,\dots,p_m$. More formally, we have:

\begin{observation3}\label{obs:itworks,bitches!}
Suppose $F = \<\z_1,\dots,\z_{\ell_m}\>$ is a finite sequence of ordinals $<\!\dlt_\a$, that $E = \<\xi_1,\dots,\xi_{\ell_m}\>$ is a finite sequence of ordinals $<\!\dlt_{\a+1}$, and that
$$M_\a \models \Phi^m(\z_1,\z_2,\dots,\z_{\ell_m},X,\psi) \ \ \ \text{and} \ \ \ M_{\a+1} \models \Phi^m(\xi_1,\xi_2,\dots,\xi_{\ell_m},X,\psi).$$
Suppose further that $x,x' \in [0,1]^{\dlt_\a}$ and $y,y' \in [0,1]^{\dlt_{\a+1}}$, and that 
$$\pi_{\z_i}(x) = \pi_{\xi_i}(y) \qquad \text{and} \qquad \pi_{\z_i}(x') = \pi_{\xi_i}(y')$$
for all $i \leq \ell$. Then
$$x \mapstoooaf x' \qquad \text{implies} \qquad y \mapstooobf y'$$
for all $j \leq m$.
\end{observation3}

By the inductive hypothesis $(\text{H}3)$, $\Phi^m(F^\uparrow_m)$ is expressible in $M_{\a+1}$. Furthermore, our choice of $\Phi^m$ and $F^\uparrow_m$ ensures $M_{\a+1} \models \Phi^m(F^\uparrow_m)$. Thus
$$M_{\a+1} \models \exists x_1,x_2,\dots,x_{\ell_m}\Phi^m(x_1,x_2,\dots,x_{\ell_m},X,\psi).$$
By elementarity, $M_{\a} \models \exists x_1,x_2,\dots,x_{\ell_m}\Phi^m(x_1,x_2,\dots,x_{\ell_m},X,\psi)$, whence there is a finite sequence $E^m = \seq{\z_i}{i \leq \ell_m}$ of ordinals in $M_\a$ such that
$$M_\a \models \Phi^m(\z_1,\z_2,\dots,\z_{\ell_m},X,\psi).$$
If $\xi \in F^\uparrow_m$, then we will denote by $\z^m_\xi$ the corresponding member of $E^m$.

For each $m$, let $k(m)$ be the least natural number with the property that for all $n \geq k(m)$ and all $j \leq m$, 
$$y^\a_n \mapstoooof y^\a_{\varphi_{p_j}(n)}.$$
This $k(m)$ exists by our inductive hypothesis $(\text{H}1)$, because $\O^m(E^m)$ is a nice open cover of $[0,1]^{\w_1}$ defined with ordinals $<\!\dlt_\a$. If $m < m'$, then $\O_{m'}$ refines $\O_m$ and there are more functions $\psi_{p_j}$ to consider; so $k(m) \leq k(m')$. In other words, the function $m \mapsto k(m)$ is non-decreasing.

We are now in a position to define the maps $q_\xi$ for $\dlt_\a \leq \xi < \dlt_{\a+1}$:
$$
q_\xi(n) = \begin{cases} 
0 & \textrm{ if $k(m) < n \leq k(m+1)$ and $\xi \notin F_m^\uparrow$,} \\
q_{\z^m_\xi}(i) & \textrm{ if $k(m) < n \leq k(m+1)$ and $\xi \in F_m^\uparrow$.}
\end{cases}
$$
Roughly, this says that $q_\xi$ assumes the behavior of its avatar function $q_{\z_\xi^m}$ on the interval between $k(m)$ and $k(m+1)$, provided some suitable avatar has already been found. As $m$ increases, $\z^m_\xi$ becomes a better and better avatar, because the formula $\Phi^m$ includes more and more information about the topology of $[0,1]^{\w_1}$ and $X$ and about the flow $\psi$.

With the $q_\xi$ thus defined, we need to check that our inductive hypotheses $(\text{H}1)$-$(\text{H}3)$ remain true at the next stage of the recursion. As before, $(\text{H}2)$ and $(\text{H}3)$ are easy to check. Indeed, $(\text{H}3)$ follows trivially from the definition of the $q_\xi$ for $\dlt_\a \leq \xi < \dlt_{\a+1}$. For $(\text{H}2)$, note that because $\seq{M_\b}{\b \leq \a+1} \in M_{\a+2}$, the above construction of the $q_\xi$, $\dlt_\a \leq \xi < \dlt_{\a+1}$, can be carried out in $M_{\a+2}$. Thus the result of this construction, namely the sequence $\seq{\Delta_{\b < \dlt_{\a+1}}q_\b(n)}{ n \in \w}$, is a member of $M_{\a+2}$ as well.

For $(\text{H}1)$, let us write $y^{\a+1}_n = \Delta_{\xi < \dlt_{\a+1}}(n)$ for each $n$. Let $\O_m$ be one of the nice open covers considered in our construction; we wish to show that $\seq{y^{\a+1}_n}{n \in D}$ is $\varphi$-like with respect to $\O_m$. Because $m$ is arbitrary, this suffices to show that $(\text{H}1)$ holds at $\a+1$.

Fix $n$ with $k(m) < n \leq k(m+1)$. Then
$$y^{\a}_n \mapstoooof y^{\a}_{\varphi_{p_j}(n)}$$
for each $j \leq m$ by the choice of $k(m)$.
From this, from Observation 3, and from our choice of $E^m$, we deduce
$$y^{\a+1}_n \mapstooocf y^{\a+1}_{\varphi_{p_j}(n)}$$
or, equivalently,
$$y^{\a+1}_n \mapstooodf y^{\a+1}_{\varphi_{p_j}(n)}$$
for each $j \leq m$. If $m' > m$ and $k(m') < n \leq k(m'+1)$, then similarly
$$y^{\a+1}_n \mapstoooef y^{\a+1}_{\varphi_{p_j}(n)}$$
for each $j \leq m$. Because $\O^{m'}$ refines $\O^m$, this implies
$$y^{\a+1}_n \mapstooodf y^{\a+1}_{\varphi_{p_j}(n)}$$
for each $j \leq m$ again. Thus this relation holds for all but finitely many $n$, namely for all $n > k(m)$. In other words, we have shown that, for every $m$.
\begin{equation}
\text{if } j \leq m \text{ then} \ y^{\a+1}_n \mapstooodf y^{\a+1}_{\varphi_{p_j}(n)} \ \text{for all } n > k(m).
\tag{$*$}
\end{equation}
That $(*)$ holds for every $m$ implies that $\seq{y^{\a+1}_n}{n \in D}$ is $\varphi$-like in $[0,1]^{\dlt_{\a+1}}$. This proves that $(\text{H}1)$ holds at $\a+1$ and completes the successor step of our recursion.

We claim that the map $Q = \Delta_{\a < \w_1}q_\a$ is as required; i.e., the sequence $\seq{Q(n)}{n < \w}$ is a $\varphi$-like sequence in $[0,1]^{\w_1}$. The argument is essentially the same as it was for the preservation of $(\text{H}1)$ at limit stages of the recursion. If $\O$ is a nice open cover of $[0,1]^{\w_1}$, then $\O$ is defined by finitely many ordinals. Thus there is some $\a < \w_1$ such that $\O$ is defined from ordinals $<\!\a$. Because $(\text{H}1)$ is true at $\a$, Observation 2 implies that $\seq{Q(n)}{n < \w}$ is a $\varphi$-like sequence for $\O$. Observation 1 tells us that it suffices to consider only the nice open covers, so $\seq{Q(n)}{n < \w}$ is $\varphi$-like. An application of Lemma~\ref{lem:psilike} completes the proof. 
\end{proof}

%%%%%%%%%%%%

%%%%%%%%%%%%


\begin{thebibliography}{99}
\bibitem{akin} E. Akin, \emph{The General Topology of Dynamical Systems}, Graduate Studies in Mathematics, American Mathematical Society, reprint edition, 2010.
\bibitem{Bandlow} I. Bandlow, ``A construction in set-theoretic topology by means of elementary substructures," \emph{Zeitschrift f\"{u}r mathematische Logik und Grundlagen Mathematik} \textbf{37} (1991), pp. 467-480.
%\bibitem{Bartosova} D. Barto\v{s}ov\'a, ``Universal minimal flows of groups of automorphisms of uncountable structures, '' \emph{Canadian Mathematics Bulletin} \textbf{56} (2013), pp. 709--722.
%\bibitem{Bel} M. Bell, ``On the combinatorial principle $P(c)$,'' \emph{Fundamenta Mathematicae} \textbf{114} (1981), pp. 149-157.
%\bibitem{Bls} A. Blass, ``Ultrafilters: where topological dynamics = algebra = combinatorics,'' \emph{Topology Proceedings} \textbf{18} (1993), pp. 33-56.
\bibitem{B&S} A. B\l aszczyk and A. Szyma\'nski, ``Concerning Parovi\v{c}enko's theorem,'' \emph{Bulletin de L'Academie Polonaise des Sciences S\'erie des sciences math\'ematiques} \textbf{28} (1980), pp. 311-314.
\bibitem{bowen} R. Bowen, ``$\w$-limit sets for axiom $A$ diffeomorphisms,'' \emph{Journal of Differential Equations} \textbf{18} (1975), pp. 333-339.
\bibitem{WRB} W. R. Brian, ``$P$-sets and minimal right ideals in $\N^*$,'' \emph{Fundamenta Mathematicae} \textbf{229} (2015), pp. 277-293.
\bibitem{Brian} W. R. Brian, ``Abstract omega-limit sets,'' to appear in \emph{Journal of Symbolic Logic}, currently available at \texttt{wrbrian.wordpress.com/research}.
%\bibitem{vDP} E. K. van Douwen and T. C. Przymusi\'nski, ``Separable extensions of first-countable spaces,'' \emph{Fundamenta Mathematicae} \textbf{111} (1980), pp. 147 - 158.
\bibitem{Dow&Hart} A. Dow and K. P. Hart, ``A universal continuum of weight $\aleph$,'' \emph{Transactions of the American Mathematical Society} \textbf{353} (2000), pp. 1819 - 1838.
\bibitem{Engelking} R. Engelking, \emph{General Topology}. Sigma Series in Pure Mathematics, 6 (1989), Heldermann, Berlin (revised edition).
\bibitem{farah} I. Farah, \emph{Analytic quotients: theory of lifting for quotients over analytic ideals on the integers,} Memoirs of the American Mathematical Society no. \text{702} (2000), vol. 148.
%\bibitem{fourheads} I. Farah, ``The fourth head of $\b\N$,'' in \emph{Open Problems in Topology II} (2007), ed. Elliot Pearl, Elsevier, pp. 145-150.
%\bibitem{geschke} S. Geschke, ``The nonexistence of universal metric flows,'' unpublished manuscript available online at \\\texttt{www.math.uni-hamburg.de/home/geschke/publikationen.html.en}.
\bibitem{geschke} S. Geschke, ``The shift on $\mathcal P(\w) / \mathrm{fin}$,'' unpublished manuscript available online at \texttt{www.math.uni-hamburg.de/home/geschke/publikationen.html.en}.
\bibitem{H&S} N. Hindman and D. Strauss, \emph{Algebra in the Stone-\v{C}ech compactification}, 2$^{nd}$ edition, De Gruyter, Berlin, 2012.
\bibitem{Hdg} W. Hodges, \emph{Model Theory}, Encyclopedia of mathematics and its applications, no. \textbf{42}, Cambridge University Press, Cambridge, 1993.
\bibitem{johnstone} P. T. Johnstone, \emph{Stone Spaces}, Cambridge Studies in Advanced Mathematics 3, Cambridge University Press, Cambridge, 1982.
%\bibitem{KPT} A. S. Kechris, V. G. Pestov, and S. Todor\v{c}evi\'c, ``Fra\"{i}ss\'e limits, Ramsey theory, and topological dynamics of automorphism groups,'' \emph{Geometric and Functional Analysis} \textbf{15} (2005), pp. 106--189.
%\bibitem{kunen} K. Kunen, \emph{Set Theory: An Introduction to Independence Proofs}. North-Holland, 1980.
\bibitem{JvM} J. van Mill, ``An introduction to $\b\omega$,'' in the \emph{Handbook of Set-Theoretic Topology} (1984), eds. K. Kunen and J. E. Vaughan, North-Holland, pp. 503-560.
\bibitem{N&U} N. Noble and M. Ulmer, ``quotienting functions on Cartesian products,'' \emph{Transactions of the American Mathematical Society} \textbf{163} (1972), pp. 329-339.
\bibitem{Parovicenko} I. I. Parovi\v{c}enko, ``A universal bicompact of weight $\aleph_1$,'' \emph{Soviet Mathematics Doklady} \textbf{4} (1963), pp. 592-595.
%\bibitem{Pestov} V. Pestov, ``On free actions, minimal flows, and a problem by Ellis,'' \emph{Transactions of the American Mathematical Society} \textbf{350} (1998), pp. 4149--4165.
%\bibitem{samuel} P. Samuel, ``Ultrafilters and the compactification of uniform spaces," \emph{Transactions of the American Mathematical Society} \textbf{64} (1948), pp. 100--132.
\bibitem{Shchepin} E. V. Shchepin, ``Real functions and canonical sets in Tikhonov products and topological groups,'' \emph{Russian Mathematical Surveys} \text{31}, no. 6 (1976), pp. 17-27.
\bibitem{shelah} S. Shelah, \emph{Proper Forcing}, Lecture Notes in Mathematics, vol. 940, Springer-Verlag, Berlin, 1982.
\bibitem{S&S} S. Shelah and J. Stepr\={a}ns, ``\pfa implies all automorphisms are trivial,'' \emph{Proceedings of the American Mathematical Society} \textbf{104} (1988), pp. 1220 - 1225.
\bibitem{velickovic} B. Velickovic, ``\oca and automorphisms of $\pwmf$,'' \emph{Topology and its Applications} \textbf{49} (1992), pp. 1-12.
\end{thebibliography}
\end{document}